\documentclass{amsart}

\usepackage{amsfonts,amssymb,verbatim,amsmath,latexsym,textcomp,amscd, hyperref}
\usepackage{epsfig,verbatim}
\usepackage{graphics,textcomp}
\usepackage{graphicx}
\usepackage{url}
\usepackage[usenames,dvipsnames]{color}
\usepackage{txfonts, pxfonts}

\usepackage{tikz}
\usetikzlibrary{through,calc,shapes,positioning}

\newtheorem{theorem}{Theorem}[section]
\newtheorem{prop}[theorem]{Proposition}
\newtheorem{lemma}[theorem]{Lemma}
\newtheorem{cor}[theorem]{Corollary}
\newtheorem{defn}[theorem]{Definition}

\newtheorem{rmk}[theorem]{Remark}
\newtheorem{qn}[theorem]{Question}

\newcommand{\ra}{\rangle}

\newcommand{\natls}{{\mathbb N}}

\newcommand{\e}{{\epsilon}}

\newcommand\CC{{\mathcal C}}

\newcommand\FF{{\mathcal F}}
\newcommand\GG{{\mathcal G}}
\newcommand\HH{{\mathcal H}}

\newcommand\LL{{\mathcal L}}
\newcommand\MM{{\mathcal M}}

\newcommand\PP{{\mathcal P}}

\newcommand\RR{{\mathcal R}}

\newcommand\TT{{\mathcal T}}

\newcommand{\CG }{{\rm Con} ^{\omega } (G, d)}

\newcommand\PMF{{\PP\kern-2pt\MM\FF}}

\newcommand\PML{{\PP\kern-2pt\MM\LL}}

\newcommand{\fsubd}{\mathrel{{\scriptstyle\searrow}\kern-1ex^d\kern0.5ex}}
\newcommand{\bsubd}{\mathrel{{\scriptstyle\swarrow}\kern-1.6ex^d\kern0.8ex}}
\newcommand{\fsubeq}{\mathrel{\raise-.7ex\hbox{$\overset{\searrow}{=}$}}}
\newcommand{\bsubeq}{\mathrel{\raise-.7ex\hbox{$\overset{\swarrow}{=}$}}}

\newcommand{\tsh}[1]{\left\{\kern-.9ex\left\{#1\right\}\kern-.9ex\right\}}

\begin{document}

\title{Controlled   Floyd Separation and Non Relatively Hyperbolic Groups}

\author{Shubhabrata Das}
\address{School of Mathematical Sciences, Department of Mathematics, RKM Vivekananda University, P.O. 
Belur Math, Howrah 711202, India}
\email{shubhabrata.gt@gmail.com}
\author{Mahan Mj}
\address{School of Mathematical Sciences, Department of Mathematics, RKM Vivekananda University, P.O. 
Belur Math, Howrah 711202, India}
\email{mahan.mj@gmail.com, mahan@rkmvu.ac.in}

\thanks{
SD is partly supported by a CSIR Junior Research Fellowship. MM is supported in part by a JC Bose fellowship.
This paper is part
of SD's PhD thesis  written under the supervision of MM. } 
\date{\today}



\begin{abstract}
We introduce the notion of  controlled Floyd separation between geodesic rays  starting at the identity in a finitely generated group  $G$.
 Two such geodesic rays are said to be Floyd separated with respect to quasigeodesics if the (Floyd) length of $c-$quasigeodesics (for fixed but arbitrary $c$)
joining points on the   geodesic rays is asymptotically bounded away from zero. This is always satisfied by Morse geodesics.
The main purpose of this paper is to furnish an example of a finitely generated group $G$ such that \\
1) all  finitely presented subgroups of $G$ are hyperbolic, \\
2) $G$ has an uncountable family of geodesic rays that are Floyd separated with respect to 
quasigeodesics,\\
3) $G$
 is not hyperbolic relative to any collection of proper subgroups.\\
4) $G$ is a direct limit of hyperbolic CAT(0) cubulated groups.\\
5) $G$ has trivial Floyd boundary in the usual sense.

On the way towards constructing $G$, we construct a malnormal infinitely generated (and hence non-quasiconvex) subgroup of a free group,
giving negative evidence towards a question of Swarup and Gitik.
\end{abstract}

\maketitle

\tableofcontents

\section{Introduction}\label{intro}  In this paper, we introduce the notion of  controlled Floyd separation between geodesic rays 
starting at the identity
 of a finitely generated group $G$.
Here, the term {\it `control'} refers  in general to a choice of a $G-$invariant subcollection of 
paths in a Cayley graph $\Gamma$ of $G$ with respect to which the Floyd distance is computed
(See Definition \ref{cdef}).
We shall primarily be interested in subcollections given by $(\kappa, \kappa)-$quasigeodesics in $\Gamma$. 
The main purpose of this paper is to furnish an example of a finitely generated group $G$ which satisfies the following properties:
\begin{enumerate}
\item[(A)] There are uncountably many geodesic rays starting  at the identity in $G$, which are  Floyd separated with
respect to $(\kappa, \kappa)-$quasigeodesics for any $\kappa > 1$. See Theorem \ref{nontrivial}. These rays actually lie in a subgroup that
is quasiconvex in a strong sense: all geodesics in the subgroup are `uniformly Morse' (Corollary \ref{morse}).
\item[(B)] $G$ is not hyperbolic relative to any collection of proper subgroups.  See Theorem \ref{nrhthm}.
\end{enumerate}
This provides a construction that may be used as negative evidence towards a question of Olshanskii, Osin and Sapir \cite[Problem 7.11]{oos}.

The group $G$ we construct is the double of a free group along an infinitely generated malnormal subgroup $K$:  $G=F\ast_K F$. Here
$K = \cup_n K_n$ is an ascending union of malnormal quasiconvex subgroups $K_n$ of $F$. The group $G$ has a number of other features:

\begin{enumerate}
\item $G$ is a graded small cancellation (and hence lacunary hyperbolic) group in the sense of Olshanskii, Osin and Sapir \cite{oos}. See Theorem \ref{g-gsc}.
\item $G$ is a direct limit of cubulated hyperbolic groups $G_n$: each $G_n$ is hyperbolic and admits a geometric action on
a non-positively curved square complex (in fact a $\mathcal{VH}$ complex in the sense of Wise \cite{wise-cbms}). See Proposition \ref{cycgeodprop}.
\item Every finitely presented subgroup of $G$ is hyperbolic (Proposition \ref{subgp}). In particular, $G$ does not contain any Baumslag-Solitar group
$BS(m,n)$.
\item Each $K_n$ is quasiconvex in a strong sense (`uniformly Morse') in $G$ (Corollary \ref{morse}).
\item The limit set of $K (= \cup_n K_n)$ in $F$ is all of the Gromov boundary of $F$ (Proposition \ref{full}).
\item In spite of the  features (A) and (B) above, $G$ has trivial Floyd boundary in the usual sense.  See Theorem \ref{floydtrivial}.
\end{enumerate}

The group $G$ and its features above illustrate that the source of triviality of the Floyd boundary of a group $G$ is quite subtle
and known sources (e.g. existence of $\mathbb{Z} \oplus \mathbb{Z} \subset G$, wideness of $G$, etc.) are inadequate to detect this.

On the way towards constructing $G$, we construct (Proposition \ref{malnprop}) a malnormal infinitely generated (and hence non-quasiconvex) subgroup of a free group,
giving negative evidence towards a question of Swarup and Gitik \cite[Question 1.8]{bestvinahp}.

A range of tools is used in proving the properties of $G$:
\begin{enumerate}
\item The theory of relative hyperbolicity and relative quasiconvexity \cite{farb-relhyp, bowditch-relhyp, hruska-agt}.
\item Graphical small cancellation theory  \cite{wise-por, wise-qpor}.
\item Graded small cancellation theory \cite{oos, olshanski-book}.
\item Tree graded spaces \cite{ds}.
\item Asymptotic cones and lacunary hyperbolic groups \cite{oos, gromov-ai}.
\end{enumerate}

\begin{comment}
\noindent {\bf A Mnemonic:} In Theorem \ref{nontrivial} we show that ideal points on the Gromov boundary of a qi-embedded free subgroup $K_2$ of $G$ are
Floyd-separated with respect to quasigeodesics.
To show this we pass to the asymptotic cone $Con^\omega G$ of $G$  which is a circle tree (Proposition \ref{cctree}). The asymptotic cone of $K_2$ is embedded
as an $\mathbb R$-tree $\TT$ in $Con^\omega G$. Theorem \ref{trivialintn} shows that the intersection of any circle in $Con^\omega G$ has to be a single point or
empty. In particular, it cannot be an arc. The reader may think of this as the statement
 `$D$'s are not allowed inside $Con^\omega G$'. 

The proof of Theorem \ref{floydtrivial} shows that in fact it is possible for a circle
in $Con^\omega G$ to touch $\TT$ at a point. Thus, `Circles tangential to  lines is possible in $Con^\omega G$'. 
\end{comment}

\subsection{History of the Problem} The notion of a Floyd function and Floyd boundary were introduced by Floyd in \cite{floyd}
and generalized by Gromov in \cite{gromov-ai} under the name of conformal boundary of a group. Gromov \cite[p. 264]{gromov-ai} and
Karlsson \cite{karlsson-free} showed that
the action of a group on its Floyd boundary is a convergence action (in fact a geometric  convergence action in the sense of Gerasimov) provided the boundary is
non-trivial, i.e. if it contains more than two (and hence infinitely many) points. In \cite[Proposition 4.28]{oos}, Olshanskii, Osin and Sapir showed that 
if $G$ is a finitely generated group whose Floyd boundary consists of at least
2 points, then all asymptotic cones of $G$ have cut points. These two facts may be regarded as evidence  towards a positive answer to the following question:

\begin{qn}  \cite[Problem 7.11]{oos} (see also \cite[p.18]{act}.) Suppose that a finitely generated group G has a non–trivial Floyd boundary.
Is $G$ hyperbolic relative to a collection of proper subgroups?
\label{mainqn} \end{qn}

  Theorem \ref{nontrivial} shows, in particular, that the following weaker question has a negative answer (see Section \ref{cfs} for precise definitions).

\begin{qn}  Suppose that a finitely generated group G has a pair of geodesic rays which are  Floyd-separated with respect to quasigeodesics.
Is $G$ hyperbolic relative to a collection of proper subgroups?
\label{mainqn2} \end{qn}

A very general and detailed treatment of relative hyperbolicity and its characterization in topological and dynamic terms has been carried out recently 
by Gerasimov and Potyagailo in a sequence of papers
\cite{gerasimov-expans-gafa,ger-floydgafa,ger-pot-qc, ger-pot-jems}. The notion of a controlled Floyd separation is related to, but different from the notion of
a Karlsson function introduced in \cite{ger-pot-qc}.

\subsection{Notation}\label{notn}

\begin{enumerate}
\item $H$ will be a subgroup of the free group $F_2$ and $h, k$ will denote elements (typically of $F_2$).
\item $(h,k)_1$ will denote the Gromov inner product of $h, k$ with respect to $1$.
 \item $|h| = d(h,1)$
\item $(h,H)_1 = {\rm max}_{k \in H} (h,k)_1$
\item $Rose (F_2)$ will denote the rose on two petals labeled by the generators  of $F_2$.
\item $K_n$ will be a free subgroup of $F_2$ of rank $n$. 
$Rose (K_n)$ will denote the rose on $n$ petals labeled by the generators  of $K_n$.
\item If $H$ is a subgroup of a (relatively) hyperbolic group $G$, the limit set of $H$ is denoted $\Lambda_H$. 
\item $T(n)$ denotes the tower function of height $n$ with base 2, i.e. $T(n) = 2^{2^{2^{2 \cdots}}}$, where the power is taken
$ n$ times.
\item $U(n)(a,b)$ is the word (element) in the free group $F_2$ on generators $a, b$
given by  $$U(n)(a,b)=[a^{(T(n)+1)}b^{(T(n)+1)}][a^{(T(n)+2)}b^{(T(n)+2)}]\cdots[a^{T(n+1)}b^{T(n+1)}].$$
\end{enumerate}

\section{Preliminaries}\label{prelim}

We shall be using two equivalent definitions regarding quasigeodesics in this paper, which we elucidate here at the outset.

\begin{defn} Let $(X,d)$ be a geodesic metric space. \\
1) A path $\sigma : I \rightarrow X$ is a $(K, \epsilon)-$quasigeodesic with $K\geq 1, \epsilon \geq 0$ if 
for all $t_1, t_2 \in I$, 
$$\frac{1}{K} |t_1-t_2| - \epsilon \leq  d(\sigma (t_1), \sigma (t_2)) \leq K |t_1-t_2| + \epsilon.$$\\
2) Alternately \cite{oos}[Section 2.4], $\sigma : I \rightarrow X$ is a $(\lambda, \epsilon)-$quasigeodesic with $\lambda\in (0, 1], \epsilon \geq 0$ if 
for all $t_1, t_2 \in I$,  
$$d(\sigma (t_1), \sigma (t_2)) \geq \lambda l(\sigma([t_1, t_2]) -\epsilon,$$
where $l(\sigma([t_1, t_2]) $ denotes the length of the subpath  from $\sigma (t_1)$ to $ \sigma (t_2)$. 
\end{defn}

We use the first definition when we want $K \geq 1$ and the second one when we want $\lambda\in (0, 1]$. The second one is used primarily in  Section \ref{gsc}.

\subsection{Relative Hyperbolicity} \label{rh} We quickly recall the notion of relative hyperbolicity introduced by Gromov \cite{gromov-hypgps}.
An equivalent notion was introduced by Farb \cite{farb-relhyp} and the equivalence of the two proven by Bowditch
\cite{bowditch-relhyp}. Equivalent definitions and closely related work
can also be found in the work of a number of authors including \cite{bowditch-relhyp, osin-mams,
dahmani-th, groves-manning, hruska-agt}.

Let $(X,d)$ be a path metric space. A collection of closed
 subsets $\HH = \{ H_\alpha\}$ of $X$ will be said to be {\bf uniformly
 separated} if there exists $\epsilon > 0$ such that
$d(H_1, H_2) \geq \epsilon$ for all distinct $H_1, H_2 \in \HH$.

 \begin{defn}  \cite{gromov-hypgps} 
For any geodesic metric space
$(H,d)$, the {\em hyperbolic cone} 
$H^h$ is the metric space
$H\times [0,\infty) = H^h$ equipped with the
path metric $d_h$ obtained as follows: \\
1) $d_{h,t}((x,t),(y,t)) = 2^{-t}d_H(x,y)$, where $d_{h,t}$ is the induced path
metric on $H\times \{t\}$.  Paths joining
$(x,t),(y,t)$ and lying on  $H\times \{t\}$
are called {\em horizontal paths}. \\
2) $d_h((x,t),(x,s))=\vert t-s \vert$ for all $x\in H$ and for all $t,s\in [0,\infty)$, and the corresponding paths are called
{\em vertical paths}. \\
3)  for all $x,y \in H^h$,  $d_h(x,y)$ is the path metric induced by the collection of horizontal and vertical paths. \\
\end{defn}

A similar construction of a `combinatorial' horoball was carried out by Groves-Manning in \cite{groves-manning}.

\begin{defn}
Let $\Gamma$ be the Cayley graph of a group $G$ and let $\HH_0$ be a finite collection of  subgroups of $G$.
let $\HH$ be the collection of all left cosets of elements of $\HH_0$.
$G$ is said to be  hyperbolic relative to $\HH$ in the sense of Gromov, if $\GG (\Gamma, \HH)$,  
obtained by attaching the hyperbolic cones
$ aH^h$ to $aH \in \HH$  by identifying $(z,0)$ with $z$
for all $aH\in \HH$ and $z \in aH$,
 is a complete hyperbolic metric space. The collection $\{ aH^h : H \in \HH \}$ is denoted
as ${\HH}^h$. The induced path metric is denoted as $d_h$.

The boundary $\partial \GG (\Gamma, \HH)$ is called the {\bf Bowditch boundary}.
\end{defn}

The following Theorem
of Bowditch \cite{bowditch-relhyp} shall be useful (see also \cite{mahan-relrig}).

\begin{theorem}  \cite{bowditch-relhyp} Let $G$ be a hyperbolic group without torsion. Then $G$ is hyperbolic relative to $\PP$ if and only if 
\begin{enumerate}
\item Each $P_i$ is quasiconvex in $G$.
\item for all $P_i, P_j \in \PP$ and $g \in G$,
$gP_ig^{-1} \cap P_j = \{1\}$ unless $i=j$ and $g \in P_i$. 
\end{enumerate}

In particular if $\PP$ has exactly one subgroup $P_1$, and $G$ is hyperbolic relative to $P_1$, the latter is malnormal quasiconvex in $G$.
\label{malnrh} \end{theorem}

\subsubsection{Relative Quasiconvexity}
In \cite{hruska-agt} Hruska gives a number of equivalent criteria for relative quasiconvexity for subgroups of relatively hyperbolic groups.
We give one of these below to fix notions.

\begin{defn}\cite{hruska-agt} Let $G$ be  hyperbolic relative to a collection $\PP$ of parabolic subgroups. 
A subgroup $H \subset G$ is \emph{relatively quasiconvex} if the following holds.
Let $\mathcal{S}$ be some (any) finite relative generating set for $(G,\PP)$,
and let $\mathcal{E}$ be the set containing all the elements of  $P_i \in \PP$ (for all $i$).
Let $\Gamma$ be the Cayley graph of $G$ with respect to the generating set  $\mathcal{S} \cup \mathcal{E}$
with all edges of length one.
Let $d$ be some (any) proper, left invariant metric on~$G$.
Then there is a constant $D_0=D_0(\mathcal{S},d)$ such that
for each geodesic $\bar{c}$ in $\Gamma$
connecting two points of $H$,
every vertex of $\bar{c}$ lies within a $d$--distance $D_0$ of $H$.
\end{defn}

It is shown in \cite{hruska-agt} that the above definition is independent of the choice of
finite relative generating set $\mathcal{S}$ and the choice of proper metric $d$.

%%%%%%%%%%%%%%%%%%%%%%%%%%%%%%%%

\begin{theorem}\cite{hruska-agt}[Theorem 1.2]
\label{hruska-intn}
Let $G$ be a countable group that is relatively hyperbolic with respect
to a finite family of subgroups $\PP=\{P_1,\dots,P_n\}$.
\begin{enumerate}
\item If $H\subset G$ is relatively quasiconvex, then $H$ is relatively
hyperbolic with respect to a natural induced collection of subgroups.
\item If $H_1,H_2 \subset G$ are relatively quasiconvex, then
$H_1 \cap H_2$ is also relatively quasiconvex.
\end{enumerate}
\end{theorem}

\begin{theorem}\cite{hruska-agt}[Theorem 1.5]
\label{hruska-qi}
Let $G$ be a finitely generated relatively hyperbolic group
and let $H$ be a finitely generated subgroup.
If $H$ is undistorted in $G$, then $H$ is relatively quasiconvex.
\end{theorem}

Combining the equivalence of characterizations in \cite{hruska-agt} with a Lemma of Yang we get

\begin{lemma}\cite{yang-relhyp}[Lemma 2.6] \label{yang}
Let $H$ be relatively quasiconvex in a relatively hyperbolic group $G$ such
that $|\Lambda(H)| \geq 2$. Then for any subgroup $H \subset J
\subset G$ satisfying $\Lambda(H)=\Lambda(J)$, we have that $H$ is
of finite index in $J$. In particular, $J$ is relatively 
quasiconvex.
\end{lemma}

\subsection{Floyd Boundary and Controlled Floyd Separation}\label{fb}

\begin{defn}
A function $f:\mathbb N\to\mathbb R$ is  a {\bf Floyd  scaling function}
if 
\begin{enumerate}
 \item $\sum_{n\geqslant0}f_n<\infty$,
\item There exists a positive $\lambda$ such that $1\geqslant f_{n+1}/f_n\geqslant\lambda$ for all $n{\in}\mathbb N$.
\end{enumerate}
\end{defn}

In this paper we shall refer to a Floyd  scaling function simply as a {\bf scaling function}.

Let $G$ be a group with symmetric generating set $S$ and $\Gamma = \Gamma (G,S)$ denote the Cayley graph of $G$ with respect to $S$
and let $d$ denote the word-metric.
 
\begin{defn}
Let $f$ be a scaling function. Given an edge $[x,y]$ of $\Gamma$, define its Floyd length to be $f( d([x,y], 1))$. 
The resulting {\bf path metric} on $\Gamma$  is called the {\bf Floyd metric } with respect to the scaling function $f$ and is denoted as $d_f$.

A geodesic in $(\Gamma, d_f)$ is called a {\bf Floyd geodesic }.

The metric completion $\overline{(\Gamma, d_f)}$ is called the {\bf Floyd completion} and 
$\partial_f\Gamma = \overline{(\Gamma, d_f)}\setminus
\Gamma$ is called the {\bf Floyd boundary}. For $f(n) = \lambda^n$ with $\lambda \in (0,1)$, $d_f$ will also be denoted as $d_\lambda$ and 
$\partial_f\Gamma$ will  also be denoted as  $\partial_\lambda\Gamma$.
\end{defn}

We give a somewhat different but equivalent description of the Floyd boundary that will be better adapted to the applications we have in mind.

\begin{lemma} \label{geodfg} Let $f$ be a (Floyd) scaling function for $\Gamma$. Let $r$ be a geodesic ray in $\Gamma$ starting at $1\in \Gamma$. 
Then $r$ is a Floyd geodesic in $(\Gamma, d_f)$. \end{lemma}

\begin{proof} Let $p_m, p_n$ be points on $r$ such that $d(1,p_m) = m $ and $d(1,p_n) = n $.
Let $r_{mn}$ be the subsegment of $r$ joining $p_m, p_n$. Then any path $\sigma$ from $p_m$ to $p_n$ must have, for $m \leq i < n$,
 at least one edge $E_i$
joining a point $q_i$ to a point $q_{i+1}$ where $d(q_i, 1) = i$ and $d(q_{i+1}, 1) = i+1$ since any such path must travel between the $i-$ and $(i+1)-$spheres
at least once. Hence $l_f(\sigma) \geq \sum_m^{n-1} l_f(E_i) = l_f(r_{mn})$. \end{proof}.

\begin{defn} The  set of geodesic rays in $\Gamma$ starting at the identity will be called the {\bf preboundary} of $\Gamma$. \end{defn}
Since geodesic rays starting at the identity in $\Gamma$ are also Floyd geodesics, we shall denote the preboundary as $Pre(\partial)_f \Gamma$, emphasizing the scaling
function $f$.

\begin{rmk}\label{trivialrmk}
A sequence of points tending to infinity along a geodesic ray $r$ starting at the identity is necessarily Cauchy in the Floyd metric and hence $Pre(\partial)_f \Gamma$
can be identified with the ideal points of these rays. If $p$ denote such an ideal point, we shall also denote the ray from $1$ corresponding to $p$ as $[1,p)$. 
It follows from the definition of Floyd boundary that the Floyd boundary of $\Gamma$ is a singleton set if and only if for any pair $p, q \in Pre(\partial)_f \Gamma$,
and  sequences $p_n \rightarrow p$,   $q_n \rightarrow q$,  there exist paths $\sigma_n$ in $\Gamma$ joining $p_n, q_n$ such that $l_f (\sigma_n) \rightarrow 0$
as $n \rightarrow \infty$. \end{rmk}

The following Theorem will be used later.

\begin{theorem} \cite[Proposition 3.4.6]{ger-floydgafa}
 Let
  $G$ be    relatively hyperbolic with respect to a collection $\mathcal P$ of subgroups.
Then there exists $\lambda \in (0,1)$ such that
the identity map $G\to G$ extends to a continuous equivariant map $F$
from the Floyd completion $\overline{(\Gamma, d_f)}$ with respect to $ f(n) = \lambda^n$
to the Bowditch completion of $G$ with respect to $\mathcal P$.
In particular, the Floyd boundary of $G$ is non-trivial.\label{gerfloydgafa}
\end{theorem}

\subsubsection{ Controlled Floyd Separation}\label{cfs}
It is at this stage that we introduce an essentially new ingredient.

\begin{defn}\label{cdef}
Let $f$ be a scaling function and let $l(\sigma)$ denote the Floyd length of a path $\sigma$ with respect to the scaling function $f$. 
Let $\Lambda$ denote a  $G-$invariant collection of paths in $\Gamma$ such that for every $p, q\in \Gamma$, the geodesic  joining
 $p, q$ belongs to $\Lambda$. Let $\Lambda(p,q)$ denote the subcollection of paths in $\Lambda$ joining $p, q$.
We let
$d_{f, \Lambda}(p,q) = inf_{\sigma \in \Lambda(p,q)} (l(\sigma))$.

If $\Lambda$ consists of all $(\kappa, \kappa)$-quasigeodesics then  we also denote $d_{f, \Lambda}(p,q)$ by $d_{f, \kappa}(p,q)$. 

For $p, q \in Pre(\partial)\Gamma$ we define $$d_{f, \Lambda}(p,q) := {\rm liminf}_{p_n\rightarrow p, q_n\rightarrow q} \, d_{f, \Lambda}(p_n,q_n),$$ where $p_n \in [1,p)$
and  $q_n \in [1,q)$ are Cauchy sequences converging to $p, q$ respectively. \end{defn}

Note that for $p, q \in Pre(\partial)\Gamma$, the quantity $d_{f, \Lambda}(p,q)$ is independent of the Cauchy sequences $\{ p_n\}, \{ q_n \}$.
If $\Lambda$ consists of $(\kappa, \kappa)$-quasigeodesics and $p, q \in Pre(\partial)\Gamma$, then $d_{f, \Lambda}(p,q)$ is also denoted as $d_{f, \kappa}(p,q)$.

\begin{defn}
If $d_{f, \Lambda}(p,q)>0$ for $p, q \in Pre(\partial)\Gamma$, we shall say that $p, q$ are {\bf Floyd separated with respect to $\Lambda$}. 
If $d_{f, \kappa}(p,q)>0$ for $p, q \in Pre(\partial)\Gamma$ and for all $\kappa$, we shall say that $p, q$ are {\bf Floyd separated with respect to quasigeodesics}. \end{defn}

\begin{rmk} (Gerasimov) Note that  $d_{f, \Lambda}(p,q)$ is not necessarily a metric as the triangle inequality need not be satisfied. \label{gerrmk} \end{rmk}

In future, if we want to emphasize that the  Floyd length of a path $\sigma$ is being computed with respect to a Floyd function $f$
and  $\Lambda$ consists of $(\kappa, \kappa)$-quasigeodesics, we shall have occasion to use the suggestive notation $l_f^\kappa (\sigma)$
in place of $l_f (\sigma)$ only to remind ourselves that $\Lambda$ consists of all $(\kappa, \kappa)-$quasigeodesics and hence $\sigma$ itself is a 
$(\kappa, \kappa)-$quasigeodesic.

\begin{rmk} \label{dah}  Let $\gamma$ be a bi-infinite Morse geodesic in the Cayley graph $\Gamma$ of a group 
(or more generally a space $X$). Then the end-points of $\gamma$ in  $Pre(\partial)\Gamma$
will necessarily be Floyd separated with respect to quasigeodesics. In particular, hyperbolically embedded subgroups \cite{dgo} of mapping class groups,
$Out (F_n)$ and so on will have boundary points
that are Floyd separated with respect to quasigeodesics.
\end{rmk}

Floyd separation of $p, q \in Pre(\partial)\Gamma$ with respect to quasigeodesics is, in fact,  weaker than the existence of a Morse geodesic joining $p, q$.
To see this, let $G$ be hyperbolic relative to a $\mathbb{Z} \oplus \mathbb{Z} $, e.g. the fundamental group of  a finite volume
hyperbolic 3-manifold with one cusp. Let $\gamma$ be  a bi-infinite geodesic that is a union of two rays $\gamma_1, \gamma_2$
such that $\gamma_1$ is Morse and $\gamma_2$  lies close to a peripheral flat. Then $\gamma$ is not Morse but $p, q \in Pre(\partial)\Gamma$
are Floyd-separated with respect to quasigeodesics.
However, the proof of Floyd separation  with respect to quasigeodesics
that we provide in this paper (Theorem \ref{nontrivial}) actually furnishes the stronger conclusion of existence of a large collection of Morse (quasi)geodesics
(Corollary \ref{morse}). 

Further, we believe that the source of Floyd separation in this paper comes from a subgroup ($K_2 \subset G$ in the example of Section \ref{eg})
which is {\it not } hyperbolically embedded in $G$ (See Question \ref{hypembed} at the end of the paper).

\subsection{Graphical Graded Small Cancellation}\label{wgsc} In this Subsection we give a  discussion of graphical Graded Small Cancellation
following Wise \cite{wise-por, wise-qpor}. Our definitions
are equivalent to those occurring in  \cite{wise-por, wise-qpor}.
 Let $B$ be a metric graph with edges of length one.  Let $A$ be a (non-metric)
topological graph (i.e. equipped with only a CW complex structure). A map $\phi : A \rightarrow B$ is {\em combinatorial} if 
\begin{enumerate}
\item $\phi$  takes vertices 
of $A$ to     vertices 
of $B$.
\item $\phi$ restricts to an immersion on the interior of every edge of $A$ (where an
immersion is a local injection).
\end{enumerate}

A combinatorial map $\phi$ induces a metric graph structure $A_\phi$ on $A$ after subdividing $A$
(e.g. by pulling back the CW structure on $B$ via $\phi$) making $\phi$ a simplicial  isometry restricted to the interior of every edge of $A_\phi$.

A {\em path} $\phi: P\rightarrow B$ is a combinatorial map where $P$
is a real interval.          $|P|$ will denote the length of $P$, i.e. the number of edges in the graph $P_\phi$.    
When defined, $P Q\rightarrow B$ will denote the concatenation of $P\rightarrow B$ and $Q\rightarrow B$.     Let    
$P^{-1}$ denote the inverse of $P$.       In this paper, there will be a     correspondence between
combinatorial paths and words, and so the notation $|W |$ for the length of a word,
agrees with the notation $|W |$ for the length of the corresponding path.  

A (non-metric) graph $A$ is {\bf graded} if it is equipped with a
function $g: Edges(A)\rightarrow \natls$, where $g(J)$ is called  the grade of $J$.
   Let $\phi: A\rightarrow B$ be a combinatorial  map where $A$ is a graded graph and $B$ is a metric graph. Let $A_\phi$ denote the induced metric graph structure on
$A$.
Let $J$ be a 1-cell of the (non-metric) graph
 $A$ and $J_\phi$ the induced metric graph structure on it. $i_J$ denotes the inclusion of $J$ in $A$.

\begin{defn}       A combinatorial path $\alpha: P\rightarrow J $ is a piece of $J$ of {\bf grade $n$} provided  the following
 hold:
\begin{enumerate}
\item  There exists a path $\beta: P\rightarrow J^\prime$ where $J^\prime$ is an edge of $A$ with $g(J^\prime ) = n$ and $n$ is the least such natural number.
\item  $\alpha: P\rightarrow J$ and $\beta: P\rightarrow J^\prime$ represent distinct paths in $A$, i.e.
$i_J (\alpha(P))$ and $i_{J^\prime} (\beta(P))$ are distinct edge-paths in $A_\phi$.
 \item    $\alpha: P\rightarrow J$ and $\beta: P\rightarrow J^\prime$ project to the same path in $B$, i.e. 
$\phi(i_J (\alpha(P)))$ and $\phi(i_{J^\prime} (\beta(P)))$ give the same edge path in $B$.
\item $P$ is maximal with respect to the above conditions.
\end{enumerate}
\end{defn}

\begin{defn}  \cite{wise-qpor}[Definition 2.2]
 Let $\phi: A \rightarrow B$ be a  combinatorial  map from a (non-metric) graded graph $A$ to a metric graph $B$. $\phi$
satisfies the {\bf graded $c(p)$ condition}
provided that for each edge $J$ of the  (non-metric) graded graph $A$, and  for any expression of $ J_\phi$ as the
concatenation $P_1, \cdots, P_r$ of pieces of $J$ of grade $\leq g(J)$, at least $p$ of these pieces  have grade equal to $g(J)$.
\end{defn}

\begin{theorem} \cite{wise-qpor}[Theorem 1.7] Let $\phi : A \rightarrow B$ be a 
combinatorial map between a graded (non-metric) graph $A$ and a metric graph $B$. 
If $\phi$ satisfies the graded $c(3)$ small-cancellation condition then 
$\phi$ is $\pi_1$-injective. If $\phi$ satisfies the graded $c(5)$ small-cancellation condition
then $\phi_\ast (\pi_1 A) \subset \pi_1 B$ is malnormal.\label{wise-gsc} \end{theorem}

\section{The Main Example}\label{eg}  The main  example is constructed in two steps:

\begin{enumerate} 
\item We first construct a malnormal infinitely generated subgroup $K=K_\infty$ of $F_2$ satisfying a number of properties.
\item Then we double $F_2$ along  $K$ to get $G(\infty)=F_{2}*_{K_\infty} F_{2}$.
\end{enumerate}

\subsection{The Construction of $ K$}\label{constrn}
We shall first construct a sequence of malnormal subgroups $\{K_{n}\} \subset F_2 = F_2(a,b)$ by choosing a sequence of elements $g_n$ 
satisfying a suitable small cancellation condition such that 

\smallskip

\begin{enumerate} 
\item No proper power of $a$ or $b$ is an element of $K_n$ and hence $K_n$ is a proper infinite index subgroup.
\item $K_{n+1}=<K_{n}, g_{n+1}>$ where $g_{n+1} \in F_{2} \setminus K_{n}$.
\item $(g_{n+1},K_n)_1 = {\rm min}_{h \in (F_{2} \setminus K_{n})} (h, K_n)_1$.
\item $g_{n+1}$ is cyclically reduced.
\end{enumerate}

\smallskip

We start the induction with $K_1 := <g_1>-$ a malnormal cyclic subgroup of $F_2$; in particular $g_1$ is not a proper power.
We shall define $K=K_\infty = \bigcup_n K_n$ and ensure that $K$ is of infinite index in $F_2$. 
The choice of $g_n$ satisfying a suitable small cancellation condition will ensure a few things:

\smallskip

\begin{enumerate} 
\item $K=K_{\infty}$  is malnormal in $F_2$.
\item The limit set $\Lambda K$ of  $K$ in $\partial F_2$ is all of  $\partial F_2$.
\item $K$ is of infinite index in $F_2$.
\end{enumerate}

\smallskip

Recall from the Introduction that

\smallskip

\begin{enumerate} 
\item $T(n)$ denotes the tower function of height $n$ with base 2, i.e. $T(n) = 2^{2^{2^{2 \cdots}}}$, where the power is taken
$ n$ times.
\item $U(n)(a,b)$ is the word (element) in the free group $F_2$ on generators $a, b$
given by  $$U(n)(a,b)=[a^{(T(n)+1)}b^{(T(n)+1)}][a^{(T(n)+2)}b^{(T(n)+2)}]\cdots[a^{T(n+1)}b^{T(n+1)}].$$
\end{enumerate}

\smallskip

We choose $g_{n+1}$ inductively  to be of the form $g_{n+1} = h_{n+1} W_{n+1},$ where $h_{n+1}$ satisfies the following:

\smallskip

\begin{enumerate} 
\item $(h_{n+1},K_n)_1 = {\rm min}  (r, K_n)_1$, where the minimum is taken over geodesic rays $r$ in $F_2$ starting at $1$ and $h_{n+1}$
is an initial segment of $r$.
\item $|h_{n+1}|=(h_{n+1},K_n)_1 +1$,
\end{enumerate}

\begin{comment}
Here we have omitted, for convenience, the generators $a, b$ in the expression for the words $g_i, h_i$.
\end{comment}

\smallskip

and $W_{n+1}$ is of the form $W_{n+1} = c_{n+1} U(N(n)) d_{n+1}$ where

\smallskip

\begin{enumerate} 
\item[(A)] $N(n) $ is at least twice the maximum value of $m$ for which $ a^{m}$ or $b^{m}$ is a subword of some freely reduced word $w_i \in K_n$.
\item[(B)] $N(n) > |h_{n+1}|$.
\item[(C)] $c_{n+1} , d_{n+1}$ are chosen from $\{ a, b, a^{-1},  a^{-1}\}$ to ensure that $g_{n+1}$ is cyclically reduced.
\end{enumerate}

\smallskip

Define $$ K= K_\infty = \bigcup_n K_n.$$ By the choice of $g_i$, it follows that $K$ is of infinite index in $F_2$ since $K$ is infinitely generated free.

\subsection{$ K$ is malnormal}\label{maln}

\begin{lemma} $K_n \subset F_2$ is malnormal. \label{malnlemma} \end{lemma} 
\begin{proof} It suffices by Theorem \ref{wise-gsc} to show that the map $\phi_n : Rose(K_n) \to Rose(F_2)$ between the roses corresponding to $K_n, F_2$
satisfies the graded $c(5)$ condition. 

The $i$th petal (labeled by the generator $g_i$) of $Rose(K_n)$ is assigned grade $i$, $(i=1 \cdots n)$. Then by Conditions (A) and (B) in the construction of 
$W_{n+1}$ above, pieces of grade $n$ are of the form $a^mb^m$ or  $b^ma^{m+1}$. Since the number of such pieces is at least $5$, 
the map $\phi_n : Rose(K_n) \to Rose(F_2)$ 
satisfies the graded $c(5)$ condition. 
\end{proof}

\begin{prop} 
$K=K_\infty \subset F_2$ is malnormal. \label{malnprop} \end{prop}

\begin{proof} 
First observe that a nested union of malnormal subgroups is malnormal. Suppose $H = \bigcup_i H_i$, where $H_i \subset H_{i+1}$ and each $H_i$ is malnormal in $G$.
For some $h\in H, h \neq 1$ and $g \in G$ suppose that $ghg^{-1} = h_1 \in H$. Then there exists $n$ such that $h, h_1 \in H_n$ forcing $g \in H_n$ since $H_n$ is malnormal
in $G$. Hence $g \in H$ and $H$ is malnormal in $G$.

Since $K= K_\infty = \bigcup_n K_n$, and each $K_n$ is malnormal in $F_2$
by Lemma \ref{malnlemma}, it follows that $K$ is malnormal in $F_2$.\end{proof}

\subsection{$ K$ has full limit set}\label{fls}

\begin{prop} $\Lambda K = \partial F_2$. \label{full} \end{prop}

\begin{proof} This shall follow from the fact that in our construction, $(g_{n+1},K_n)_1 =(h_{n+1},K_n)_1 = {\rm min}_{h \in (F_{2} \setminus K_{n})} (h, K_n)_1$.

It suffices to show that for every $g \in F_2$, there exists $n$ and $h \in K_n$ such that $g \in [1,h]$, i.e. $g$ lies on the geodesic from $1$ to $h$.
Suppose not.  Then  $g \notin K_n$ for all $n$. Further, $(g,K_n)_1 \leq |g|$ for all $n$. Hence $(h_i,K_n)_1 \leq |g|$ for all $i \geq 2$.
Hence $|h_i| \leq |g|+1$ for all $i \geq 2$. This forces $h_i$ to range in a finite set -- a contradiction. \end{proof}

\subsection{Doubling} Define
\begin{enumerate}
\item $G(n)=F_{2}(a,b)*_{{K_n}(a,b) = {K_n}(x,y)} F_{2}(x,y)$.
\item $G=G(\infty)=F_{2}(a,b)*_{K_\infty(a,b) = {K_\infty(x,y)}}F_{2}(x,y)$.
\end{enumerate}

%%%%%%%%%%%%%%%%%%%%%%%%

In Section \ref{gsc} we shall need to reindex $G(n)$ and define $G_n := G(n+j_0)$ for a fixed $j_0$; hence we ask the reader's
indulgence in this slight incongruity of notation.

We shall  need the following fact from \cite{BF, mitra-ht}:

\begin{theorem} Let $G_0$ be a hyperbolic group and $H$ a malnormal quasiconvex subgroup. Then $G=G_0*_HG_0$ is hyperbolic and $H$ is quasiconvex in it. \label{dh}
\end{theorem}

The following is an immediate consequence:

\begin{prop} $G(n)$ is hyperbolic and $K_n$ is quasiconvex in it. \label{doublehyp}
\end{prop}

\begin{proof}  It suffices by Theorem \ref{dh} to show that $K_n$ is malnormal quasiconvex. Malnormality is the content of Lemma \ref{malnlemma}.
Since $K_n$ is finitely generated in $F_2$, it is quasiconvex.
\end{proof}

\subsubsection{Geodesics in $G$} Fix a symmetric generating set $\{ a, b, x, y, a^{-1}, b^{-1}, x^{-1}, y^{-1}\}$ of $G(n), G$.
We first fix notation.
Let $a, b$ denote the generators of the first copy of $F_2$ and let $x, y$ denote the generators of the second copy of $F_2$. Let $\Gamma$ denote the Cayley
graph of $G$ with respect to (the symmetrization of) $a, b, x, y$. Let $\Gamma_1$ denote the (sub) Cayley graph of the subgroup $F_2(a,b) \subset G$. Similarly
let $\Gamma_2$ denote the (sub) Cayley graph of the subgroup $F_2(x,y) \subset G$. 
Let $i_1, i_2$ denote the inclusion maps of $\Gamma_1$, $\Gamma_2$ respectively.
There exists  natural projection maps  $\Pi_j: \Gamma \to \Gamma_j$ identifying the generators $a, b$ with $x, y$ respectively. 

\begin{lemma} $F_2(a,b) $ and $F_2(x,y)$ are isometrically embedded in $G(n)$ and $G$. \label{ret} \end{lemma}
\begin{proof}
Clearly $d(\Pi_j(u), \Pi_j(v)) \leq d(u, v)$.
Since $\Pi_j \circ i_j$ is the
identity map (on the first or second $F_2$ respectively) it follows that  $F_{2}(a,b)$ and $F_{2}(x,y)$ are isometrically embedded in $G$. \end{proof}

We shall need the following normal form Lemma (cf. \cite[p. 285-6]{ls})
for free products with amalgamation. For the group  $B \ast_A C$, let $i, j$ denote the inclusion maps
of $A$ into $ B, C$ respectively.

\begin{lemma} \label{normal} 
Every element $g$ of $B \ast_A C$ which is not in the image of $A$
can be written in the {\bf normal form}
                                     $  v_1 v_2 \cdots v_n$
where the terms $v_k$ lie in $B - i(A)$ or $C - j(A)$ and alternate between these
two sets. The length $n$, called the {\bf normal length} of the group element, is uniquely determined and two such expressions
$  v_1 v_2 \cdots v_n$ and $  w_1 w_2 \cdots w_n$ give the same element of  $B \ast_A C$ iff there are
elements $a_1 , \cdots , a_{n−1} \in A$ so that
                                   $ w_k = a_{k-1} v_k a_k^{-1}$.
(where $ a_0 = a_n = 1$).\end{lemma}

\begin{defn}
We shall say that a normal form word $g=v_1 v_2 \cdots v_n$ is in {\bf reduced normal form} if 
for each $i$, the concatenation $v_iv_{i+1}$ is a geodesic word, i.e. amongst all representatives of the form $(v_ia_i) (a_i^{-1}v_{i+1})$ with each of 
$(v_ia_i)$ and $ (a_i^{-1}v_{i+1})$ replaced by geodesic words, $v_iv_{i+1}$ is shortest. \label{rednormal} \end{defn}

\begin{prop} \label{normalgeod} Given an element $g\in G(n)$  there exists a normal form geodesic word  $  v_1 v_2 \cdots v_n$ representing $g$.
\end{prop}

\begin{proof} Let $T$ denote the Bass-Serre tree of the splitting. Observe first that a normal form word as defined above projects to a path in $T$ that does not return
to a vertex after leaving it. 

Now, let $u$ be any word that projects to a path in $T$ having a non-trivial subpath $\sigma$ that starts and ends at the same vertex of $T$.  Let  $u'$ 
be the subword of  $u$  corresponding to $\sigma$. Then $u'$ corresponds to an element of either $F_2(a,b)$ or $F_2(x,y)$.
Replace $u'$ by the geodesic word in  $F_2(a,b)$ or $F_2(x,y)$ representing it. By Lemma \ref{ret}, it follows that the resulting word is at most as long as $u$.
Continuing this process we finally obtain a normal form word whose length is at most that of $u$.
\begin{comment}

begin{enumerate}
\item The terms $v_i$ lie in $F_2(a,b) - i_1(K_n)$ or $F_2(x,y) - i_2(K_n)$ and alternate between these.
\item Each $v_i$ is a geodesic word in $F_2(a,b)$ or $F_2(x,y)$.
\item If $v_i = u_i a_i$ and $v_{i+1} = b_{i+1}u_{i+1}$ as  words and $b_{i+1} = a_i^{-1} $ as group elements of $K(n)$, then  $b_{i+1} = a_i^{-1} $ 
\end{enumerate}\

Now suppose we have two normal form words $ v= v_1 v_2 \cdots v_n$ and $ w= w_1 w_2 \cdots w_n$ giving the same element of  $B ∗_A C$,
where $v$ satisfies the hypotheses of the Proposition. By Lemma \ref{normal}
there are
elements $a_1 , \cdots , a_{n−1} \in A$ so that
                                   $ w_k = a_{k-1} v_k a_k^{-1}$.
(where $ a_0 = a_n = 1$). Since $v_i$'s are already reduced (being geodesic words in a free group) it follows that $|v| \leq |w|$. The Proposition follows.
\end{comment}
\end{proof}

\subsubsection{Cyclic Geodesics in $G(n)$}\label{cycgeod}
 We shall now carefully choose a geodesic word $R_{n+1}$ in $G(n)$  denoting the new relator which when added to  the relator set
of $G(n)$ gives $G(n+1)$. $R_{n+1}$ will satisfy the condition  that all its cyclic conjugates are geodesic words in $G(n)$.

We proceed inductively to define $R_n$ of the form $w^{-1}(x,y)w(a,b)$.
Recall that 
$g_{n+1} = h_{n+1} c_{n+1} U(N(n)) d_{n+1}$, where $h_{n+1}$ is a geodesic word in $F_2$.
Let $R_n^0 = g_n^{-1}(x,y)g_n(a,b)$. $R_n$ will be obtained from $R_n^0$ after cyclic reduction.

We shall define $w_{n+1}$ to be a reduced normal form geodesic representative of $g_{n+1}^{-1}(x,y)g_{n+1}(a,b)$ in $G(n)$ 
(of normal length $2$ as per Lemma \ref{normal}) obtained by replacing a maximal subword $w_0$ of the form $w^{-1}(x,y)w(a,b)$ by a subword $w'$ of the same form
using the structure of $G(n)$ as a double along $K_n$. 
\begin{comment}
To do this, we use the inductive hypothesis that $R_i$ is of the form $w^{-1}(x,y)w(a,b)$ for all $i < n$.
\end{comment}
By Lemma \ref{normal} any normal form word representing $g_{n+1}^{-1}(x,y)g_{n+1}(a,b)$ is of the form 
$g_{n+1}^{-1}u^{-1}u(x,y)g_{n+1}(a,b)$, where $u \in K_n$. In particular, the shortest normal form representative of 
$g_{n+1}^{-1}(x,y)g_{n+1}(a,b)$ is of the form $w^{-1}(x,y)w(a,b)$ where $w(a,b) \in (F_2(a,b) \setminus K_n(a,b))$ and 
$w(x,y) \in (F_2(x,y) \setminus K_n(x,y))$. Let $R_{n+1}$ denote such a representative. 

Since 
Item(A) of the defining condition on  $U(N(n))$ demands that $N(n) $ is  at least
twice the maximum value of $m$ for which $ a^{m}$ or $b^{m}$ is a subword of some freely reduced word in $K_n$, it follows that after any reduction as above, $w(a,b)$ contains a subword of the form
$V(n)(a,b)=[a^{\lfloor \frac{(T(n)+1)}{2} \rfloor}b^{(T(n)+1)}][a^{(T(n)+2)}b^{(T(n)+2)}]\cdots[a^{T(n+1)}b^{T(n+1)}].$

Next, $g_{n+1}(a,b) g_{n+1}^{-1}(x,y)$ is already a  shortest normal form representative as it contains $b^{T(n+1)}y^{-T(n+1)}$
as  a subword which cannot be shortened further, again by Item(A) of the defining condition on  $U(N(n))$.
It follows by the same reason that every cyclic conjugate of $R_{n+1}$ is a geodesic in $G(n)$.

Thus $R_{n+1}$ satisfies the following properties:

\begin{enumerate}
\item[ Property 1:] All cyclic conjugates of $R_{n+1}$ are geodesics in $G(n)$.
\item[ Property 2:] $R_{n+1}$ is of the form $w_n^{-1}(x,y)w_n(a,b)$ where
$w(a,b)$ contains a subword of the form
$V(n)(a,b)=[a^{\lfloor \frac{(T(n)+1)}{2} \rfloor}b^{(T(n)+1)}][a^{(T(n)+2)}b^{(T(n)+2)}]\cdots[a^{T(n+1)}b^{T(n+1)}].$
\item[ Property 3:] $|R_{n+1}| = O(|V(n)(a,b)|) = O(T(n+1)^2)$
\item[ Property 4:] The largest piece of $R_{n+1}$ is of the form $b^{(T(n+1)-1)}a^{(T(n+1)}$ which has length $2T(n+1)-1$.
(since $N(n) > |h_{n+1}|$ by choice,  initial segments cannot be maximal pieces).
\end{enumerate}

The group $G$ therefore admits a presentation $G=\langle S|\RR\rangle$, where
\begin{enumerate}
\item $S = \{a,b,x,y,a^{-1},b^{-1},x^{-1},y^{-1}\}.$
\item $\RR = \bigcup_i \{ R_i\}$. 
\end{enumerate}

A more general statement is true. We learnt the proof of the following from Dani Wise. We refer the reader to \cite{wise-cbms} for details
on non-positively curved (npc) complexes.

\begin{prop} $G=F*_A F$ is a double of a finitely generated  free group $F$ along a  finitely generated subgroup $A$,
with generating set given by the union of the generating sets of the two copies of $F$. Then $G$ has cohomological dimension $2$.

Then any reduced normal form word $g$ of the form $uvw$ (i.e. of normal length 3), where $u,w$ are geodesic
words in the first factor  and $v$ is a geodesic in the second factor are geodesics in $G$.

More generally any reduced normal form word is a geodesic.
\label{cycgeodprop} \end{prop}

\begin{proof} By Theorem 7.2 of \cite{wise-cbms},
 $G$ can be represented as the fundamental group of a non-positively curved
 ($\mathcal{VH}$) square complex formed by taking two graphs $\mathcal G$ (with $F$ as fundamental group)
as vertex spaces, and an edge space that is a ${\mathcal G} \times I$  that is attached by a local isometry on each side.
It immediately follows that $G$ has cohomological dimension $2$.

Then $g$ corresponds to a  path $\sigma$ labeled $uvw$ consisting of three vertical paths labeled $u$, $v$, $w$
 joined by two horizontal edges, one from the terminal point of $u$ to the initial point of $v$ and one 
from the terminal point of $v$ to the initial point of $w$.
 
Consider a disk diagram $D$ between $\sigma$ and an arbitrary edge path $\eta$ with the same end-points.
The dual curves emanating from $v$ cannot end on $u$ or $w$ because of the hypothesis that $uv$ and $vw$ are geodesics.
The dual curves emanating from $u$ cannot end on the $w$ since then the middle path $v$
 must have dual curves that start and end on itself, contradicting that it is itself reduced.
Thus all dual curves emanating from $\sigma$ must end on $\eta$.
Thus $\eta$ has at least as many vertical edges as $\sigma$, forcing $\sigma$ to be a geodesic.
This proves the first statement.

The proof of the last statement is now completed  by a straightforward induction: We repeat the above argument after
 combining the first $(n-1)$ words into a single geodesic.
\end{proof}

\subsection{Subgroups of $G$}

\begin{prop} All finitely presented subgroups $H$ of $G$ are hyperbolic.\label{subgp} \end{prop}

\begin{proof} Let $H = \langle S_1| R_1 \rangle$ be a finite presentation. Let $S_1 =  \{ s_1, \cdots, s_n \}; 
R_1 = \{ r_1, \cdots, r_m \}$. Let $H_0 $ be the subgroup of $G(0)$ generated by $S_1$. Then there exists $n$ such that each $r_i$ is trivial  in $G(n)$
and hence in $G(m)$ for all $m \geq n$. Since $H$ is a subgroup of $G$, it follows that $H$ is a  subgroup of  $G(m)$ for all $m \geq n$. 
By the first part of Proposition \ref{cycgeodprop}, $G$ has cohomological dimension $2$.

Hence $H$ is a finitely presented subgroup of a hyperbolic group $G(n)$ of cohomological dimension $2$.
It now follows from \cite{gersubgp} that $H$ is hyperbolic.
\end{proof}

\section{Non Relative Hyperbolicity}\label{nrh}
In this Section we show that:

\begin{theorem} \label{nrhthm} $G =F_{2}*_{K} F_{2}$ is not hyperbolic relative to any proper collection of parabolic subgroups $\PP$. \end{theorem}

\begin{proof}
Suppose not, i.e. $G$ is  hyperbolic relative to a proper collection of parabolic subgroups $\PP$.

By Theorem \ref{hruska-qi}, $F_{2}(a,b)$ and $F_{2}(x,y)$ are relatively quasiconvex in $G$ and hence are relatively hyperbolic with respect to $\PP \cap F_{2}(a,b)$, 
$\PP \cap F_{2}(x,y)$ respectively. Since
$K=F_{2}(a,b)\cap F_{2}(x,y) (\subset G)$, it follows that $K$ is also relatively quasiconvex in $G$ by Theorem \ref{hruska-intn}.
But $K \subset F_{2}(a,b) \subset G$. Hence $K$ is also relatively quasiconvex in $F_{2}(a,b)$, $F_{2}(x,y)$.

By Proposition \ref{full}, $K$ has full limit set in $F_{2}(a,b)$ as well as in $F_{2}(x,y)$, i.e. $\Lambda K=\partial F_{2}(a,b) \, \,$
and $\Lambda K=\partial F_{2}(x,y)$.
By Lemma \ref{yang}, it follows that $K$ is of finite index in $F_{2}(a,b)$. This contradicts the construction of $K$ as an infinite index subgroup.

This final contradiction yields the Theorem. \end{proof}

Theorem \ref{nrh} will also follow from Theorems \ref{floydtrivial} and \ref{gerfloydgafa}, but the above proof is direct and simple.

\section{Graded Small Cancellation Structure of $G$}\label{gsc}

In this section our primary goal is to show that the group $G = F_2(a,b) *_K F_2(x,y)$ constructed in Section \ref{eg} satisfies 
a certain (technical) graded small cancellation condition \cite{oos, olshanski-book}. This will give (as usual with small cancellation theory)
a number of geometric consequences.

\subsection{Definitions and Preliminaries on Graded Small Cancellation}
\begin{defn}
A set $\mathcal R$ of words in an
alphabet is {\bf symmetrized} if for any
$R\in \mathcal R$, all cyclic shifts of $R^{\pm 1}$ are contained in
$\mathcal R$ and each $R$ is reduced. (Note that the feature that each $R$ is {\bf reduced} is incorporated as part of the definition.)\end{defn}

\begin{defn}  \cite[Definition 4.2]{oos}  \label{piece}
Let $G$ be a group with a finite symmetric generating set $S$. Let $\RR$ be a 
symmetrized set of words in $S$. Given $\epsilon >0$, a subword $U$ of a word $R\in \mathcal R$ is called an {\it
$\epsilon$-piece}  if there exists  $R^\prime \in \RR$
such that:

\begin{enumerate}
\item $R= UV$, $R^\prime = U^\prime V^\prime $, for
some words $V, U^\prime , V^\prime $,
 \item $U^\prime = YUZ$ 
for  words $Y,Z$ satisfying $\max \{ |Y|, \,|Z|\} \leq \epsilon $.
\item $YRY^{-1}\neq R^\prime $.
\end{enumerate}
\end{defn}

In the above definition $`='$ is to be interpreted as equality as elements of $G$.
In the following definition we rename condition SC of \cite{oos} as the relative small cancellation condition.

\begin{defn}\cite[Definition 4.3]{oos}\label{rscdefn} Let $\epsilon \geq 0$, $\mu\in (0,1)$, and $\rho
>0$.
A symmetrized set $\RR$ of words in a symmetrized generating set $S$ for a group $G$
satisfies the {\bf relative small cancellation condition $C(\epsilon, \mu ,\rho)$}  if
\begin{enumerate}
\item[($RSC_1$)] The words in $\mathcal R$ represent geodesics in $G$.
\item[($RSC_2$)] $|R|\geq \rho $ for any $R\in \mathcal R$.
\item[($RSC_3$)] for any $\epsilon$-piece $U$ and any  $R\in \RR$, with $U \subset R$, $|U| \leq \mu |R|$.
\end{enumerate}
\end{defn}

\begin{defn}\cite[Definitions 4.12, 4.14]{oos}\label{gscdefn} Fix  $\alpha = 10^{-2}, K=10^6$.  The presentation
\begin{equation}
G=\langle S |  \mathcal{ R}\rangle =\left\langle S\,\left|\,
\bigcup\limits_{i=0}^\infty \mathcal R_i\right.\right\ra
\end{equation}
 is a {\bf graded small cancellation presentation} if there exist
sequences $\e =(\e_n)$, $\mu
=(\mu _n)$, and $\rho =(\rho _n)$ of positive real numbers
($n=1,2\dots $) satisfying:
\begin{enumerate}
\item[($ GSC_0$)] The group $G_0=\langle S\mid \mathcal R_0\rangle $
is $\delta_0$-hyperbolic for some $\delta_0$.
\item[($GSC_1$)] $\epsilon_{n+1}>8\max\{|R|, R\in \RR_n\}=O(\rho_n)$.
\item[($GSC_2$)] $\mu_n\rightarrow 0$ as $n \rightarrow \infty$, $\mu _n\leq
\alpha$, and $\mu_n\rho_n>K\epsilon_n$ for any $n\ge 1$.
\item[($GSC_3$)] For all $n\geq 1$, $\RR_{n}$ satisfies $C(\e_n,
\mu _n, \rho _n)$ over $G_{n-1}=\left\langle S\, \left|\,\right.
\bigcup\limits_{i=0}^{n-1} \mathcal R_i\right\rangle .$
\end{enumerate}
\end{defn}

\begin{rmk} Though we have fixed $\alpha = 10^{-2}, K=10^6$ following \cite{oos} above, any sufficiently large $K$ and sufficiently small $\alpha$ suffices.\end{rmk}

We shall need the following:

\begin{lemma} \cite[Lemmas 4.13, 4.18]{oos} \label{delta} Let 
\begin{equation}
G=\langle S | \mathcal R\rangle =\langle S|
\bigcup_0^\infty \mathcal R_i\rangle
\end{equation}
 be a {\bf graded small cancellation presentation}. Then $G_n$ is $\delta_n$-hyperbolic, where $\delta_n \leq  4 max_{R\in \RR_n} |R|$.

                Any subword of any word $R_n\in \RR_n$ of length at most $|Rn |/2$ is a $(1 - o (1), 0)$-quasi–geodesic in the Cayley graph of $G$ with respect to the
generating set $S$.\end{lemma}

\subsection{Graded Small Cancellation for $G$}

\begin{theorem} The group $G=F_2 *_KF_2$ equipped with the presentation $G=\langle S|\RR\rangle$, where
 $S = \{a,b,x,y,a^{-1},b^{-1},x^{-1},y^{-1}\}$
and $\RR = \bigcup_i\{ R_i\}$ (cf. Section \ref{cycgeod} and Proposition \ref{cycgeodprop}) is a graded small cancellation group and the presentation is a  graded small cancellation presentation. 
\label{g-gsc} \end{theorem}

\noindent {\bf Proof:} The proof will proceed by checking the properties in Definitions \ref{gscdefn} and \ref{rscdefn}.

\noindent {\bf Proof of Condition $GSC_0$:} By Theorem \ref{dh}, each $G(k) =F_2 *_{K_k}F_2$  is hyperbolic and hence any $G(k)$ can be taken as the starting group $G_0$.
$k$ will be taken large enough to ensure that the parameter values $\alpha = 10^{-2}, K=10^6$ are satisfied
and will be determined during the course of the proof. We reindex if necessary by setting $G_i = G(k+i)$ for all $i$.

\noindent {\bf Proof of Condition $GSC_1$:}

\begin{enumerate}
\item $\rho_n = |R_{n+k}|$. 
\item $\epsilon_{n+1} = 10 \rho_n$. 
\end{enumerate}

Condition $GSC_1$ immediately follows.

\noindent {\bf Proof of Condition $GSC_2$:} 
It suffices to choose $\mu_n > \frac{ 10^7 \rho_{n-1}}{\rho_n}$. Recall from Property 3 satisfied by $R_{n+k}$ that $\rho_n = O((T(n+k))^2)$.
Since $\frac{(T(n+k-1))^2}{(T(n+k))^2} \rightarrow 0$ as $n\rightarrow \infty$ very fast (more than exponentially), it will be easy to make a  specific choice for $\mu_n$
during the proof of $GSC_3$ below.

\subsubsection{Proof of $GSC_3$:}
This is the crucial step where we shall prove that the set  $\RR(n+k)$ of cyclic conjugates of $R_{n+k}$ satisfies the {\bf relative
 small cancellation condition } $C(\epsilon_{n},\mu_{n},\rho_{n})$ (Definition \ref{rscdefn}) over 
 $G_{n-1}=\langle S|\RR=\underset{i=0}{\overset{n+k-1}\bigcup}\RR(i)\rangle$, for suitable choices of 
$\epsilon_{n}$, $\mu_{n}$ and $\rho_{n}$, for all $n$.

We observe that, by Proposition \ref{doublehyp} all the groups $G(i)$ are $\delta_{i}$-hyperbolic. 
We need to therefore choose parameters $\epsilon_{n}$, $\mu_{n}$ and $\rho_{n}$ for $\RR(n+k)$ 
 satisfying $C(\epsilon_{n},\mu_{n},\rho_{n})$ over 
$G_{n-1}=\langle S|\RR=\underset{i=0}{\overset{n+k-1}\bigcup}\RR(i)\rangle$.

\noindent {\bf Proof of $RSC1$:} 
 $\RR(n+k)$ consists of cyclic conjugates of exactly one relator $R_{n+k}$ all of which are geodesics by their construction (Section \ref{cycgeod}
 and Proposition \ref{cycgeodprop}). 

\noindent {\bf Proof of $RSC2$:} 
Since all cyclic conjugates of $R_{n+k}$ are geodesics (Section \ref{cycgeod}  and Proposition \ref{cycgeodprop}), all of them have the same length $\rho_n$. 

\noindent {\bf Proof of $RSC3$:} 
We will show that for the choice of $\epsilon_{n} = 10 \rho_{n-1}$ as in the Proof of Condition $GSC_1$,  the length of an $\epsilon_{n}$-piece 
in $\RR(n+k)$ has length less than $\mu_{n}|R_{n+k}|$ for  suitably chosen $\mu_{n}$. 

Let $U$ be any $\epsilon_{n}$-piece. By the form of the words $R_{n+k}$, a subword $W$
with length at least half that of $U$ is a geodesic lying
 in one of the factors $F(a,b)$ or $F(x,y)$. Further, since $U, U^\prime$ are geodesics starting and ending at most $\epsilon_n$-apart, it follows that
there exists a geodesic subword $W^\prime$ of $U^\prime$ such that the initial and final points of $W, W^\prime$ are at most $(\epsilon_n + 4 \delta_{n-1})$ apart.
Hence assume without loss of generality that $U$ is an $(\epsilon_n + 4 \delta_{n-1})$-piece lying entirely in $F_2(a,b)$.
It follows from Definition \ref{piece} that

\begin{enumerate}
\item $R= UV$, $R^\prime = U^\prime V^\prime $, for
some words $V, U^\prime , V^\prime $ and $R, R^\prime \in \RR(n+k)$,
 \item $U^\prime = YUZ$ 
for  words $Y,Z$ satisfying $\max \{ |Y|, \,|Z|\} \leq (\epsilon_n + 4 \delta_{n-1}) $.
\item $YRY^{-1}\neq R^\prime $.
\end{enumerate}

After reducing $Y^{-1}U^\prime Z^{-1}$ to reduced normal form, we obtain therefore a  piece $U_0$ (in the usual small cancellation
sense) of length at least $|U| - 2 (\epsilon_n + 4 \delta_{n-1})$.
Recall that by Property 4 of $R_{n+k}$, a maximal piece (in the usual small cancellation
sense) has length at most $2T(n+k)$. 
Hence $|U_0| < 2 T(n+k)$ and so $|U| < 2 T(n+k) + 2 (\epsilon_n + 4 \delta_{n-1})$.

\noindent {\bf Choice of $\mu_n$:}
It therefore suffices to choose $\mu_n$ such  that $2 T(n+k) + 2 (\epsilon_n + 4 \delta_{n-1}) \leq \mu_n \rho_n$. Define $\mu_n = 2max 
\{ \frac{2T(n+k)+20 \rho_{n-1}}{ \rho_n}, \frac{ 10^7 \rho_{n-1}}{\rho_n}\}$.
Since $\rho_n = O((T(n+k))^2)$, and $\delta_{n-1} = O(\rho_n)$, it follows that $\mu_n \rightarrow 0$ as $n \rightarrow \infty$.

\noindent {\bf Choice of $k$:}
It remains to choose $k$. The choice of $k$ is dictated by the requirement that $\mu_n \leq \alpha$ for all $n$. Since $\mu_m \rightarrow 0$ as $m \rightarrow \infty$
it follows that there exists $k>0$ such that $\mu_m \leq \alpha$ for all $m\geq k$. This decides the choice of $k$ and completes the proof of
Theorem \ref{g-gsc}. $\Box$

\section{Asymptotic Cone of $G=F_{2}*_{K} F_{2}$}

We refer the reader to \cite{ds} for details on tree-graded spaces and \cite{gromov-ai, oos} for details on asymptotic cones.

\subsection{Tree-graded Spaces}\label{tgac}
\begin{defn}                    Let $X$ be a complete geodesic metric space and let $\mathcal P$ be a collection of closed
geodesic subsets.  Then $ X$ is said to be   tree-graded with respect to $\mathcal P$ if the following hold:
\begin{enumerate}
 \item Any two elements of $\mathcal P$  have at most one common point.
\item Every simple geodesic triangle (i.e. a simple loop composed of three geodesics) in $X$ is contained
      in exactly one element of $\mathcal P$.
\end{enumerate}
\end{defn}

\begin{defn} \cite{oos}  If $ X$ is   tree-graded with respect to $\mathcal P$ such that each element of $\mathcal P$ 
 is  a circle of radius uniformly bounded away from $0$ and $\infty$, then $X$ is called a circle-tree.
\end{defn}

\subsection{Asymptotic Cones of Graded Small Cancellation Groups} Let $\omega$ be a non-principal ultrafilter.
We shall say that $a_{i}= O_\omega (b_i)$ if the $\omega-$limit of $\{ \frac{a_i}{b_i}\}$ is finite and non-zero. Similarly, 
$a_{i}= o_\omega (b_i)$ if the $\omega-$limit of $\{ \frac{a_i}{b_i}\}$ is zero.
\begin{defn}\label{od-vis}\cite[Definitions 4.16,4.19]{oos}
Given an ultrafilter $\omega $ and a scaling sequence $d=(d_n)$,  a sequence of real numbers $f=(f_n)$ is said to be {\bf $(\omega ,
d)$--visible} or simply asymptotically visible, if there exists a subsequence $(f_{n_i})$ of $f$ such
that $f_{n_i} = O_\omega (d_i)$.

For a group $G$ and its Cayley graph $\Gamma$, we shall say that a sequence of loops $\{ \sigma_i\}\subset \Gamma$ is asymptotically visible
if the sequences $\{ d(1, \sigma_i) \}$ and $\{ |\sigma_i| \}$ are asymptotically visible.
\end{defn}

The following is one of the main Theorems of \cite{oos}.
\begin{theorem}\label{circletree} \cite{oos}[Theorem 4.17]
Let $G$ be a group having a graded small cancellation presentation with symmetrized relator set $\RR= \cup_i\{ R_i \}$, where $\{ R_i \}$ is the grade $i$ 
symmetrized relator set.
For any ultrafilter
$\omega$, and any sequence of scaling constants $d=(d_n)$, the
asymptotic cone $\CG$ is a circle-tree. $\CG $ is an $\mathbb
R$--tree if and only if the sequence $(\rho_n)$ from Definition
\ref{gscdefn} is not $(\omega , d)$--visible.

Further any  simple loop $C$ in $\CG$ can be realized as an $\omega-$limit of translates of $\{ R_i \}$'s.
\end{theorem}

The last statement follows from
Lemma 4.23 of \cite{oos}.

\subsection{Properties of the Asymptotic Cone of $G$}
\begin{prop} $Con^{\omega}(G)$ is a circle-tree. \label{cctree}\end{prop}
\begin{proof} This follows from Theorems \ref{g-gsc} and \ref{circletree}. \end{proof}

Let $\GG$ denote $Con^{\omega}(G)$.
\begin{prop} $i: K_2 \rightarrow G$ induces a bi-Lipschitz embedding $\hat{i}: Con^{\omega}(K_{2}) \hookrightarrow \GG$. \label{blemb} \end{prop}
\begin{proof} Since $K_2 \hookrightarrow F(a,b)$ is a quasi-isometric embedding and $F(a,b) \hookrightarrow G $ is an isometric embedding by Lemma \ref{ret}, it follows that
$i: K_2 \rightarrow G$ is a  quasi-isometric embedding. The result follows. \end{proof}

\subsubsection{Limit Circles in $Con^{\omega}(G)$ intersect $Con^{\omega}(K_{2})$ trivially} 
The asymptotic cone $\mathcal{C}^{\omega}(G,\{d_n\})$ of the graded small cancellation 
group $G$ is a circle-tree with limit circles of uniformly bounded radius. 
Let $\CC$ denote the collection of circles in  $\mathcal{C}^{\omega}(G,\{d_n\})$.

The asymptotic cone $\mathcal{C}^{\omega}(K_2,\{d_n\})$ of the free group 
$K_2$ (constructed at the second stage of the inductive construction of $K$) is an $\mathbb{R}$-tree $\TT$. The next Theorem is crucial in showing that
$K_2$ is quasiconvex in $G$ in a strong sense. This, in turn will lead to the conclusion that any pair of distinct geodesic rays in $K_2$ starting at the origin
are Floyd separated with respect to quasigeodesics. 

\begin{theorem} For any $C \in \CC$, the intersection $C\cap \TT$ is either empty or contains a single point. \label{trivialintn} \end{theorem}

\begin{proof} Suppose not. Since $\GG$ is a circle-tree by Theorem \ref{circletree}, it follows that
there exists a bigon $\sigma$ consisting of the union of two arcs:
\begin{enumerate}
\item  a geodesic  $\overline{p q}$ in $\TT$.
\item an embedded path $\gamma$ in $\GG$ joining $p, q$ with interior disjoint from $\overline{p q}$ .
\end{enumerate}

Since all elements of $\CC$ can be obtained as $\omega-$limits of translates  of elements of $\RR$
by Theorem \ref{circletree}, we can assume that there exists a sequence of translates  $\{ g_{i_n}R_{i_n} \}$ whose $\omega-$limit
is $\sigma$. For the purposes of this proof, reindex by setting $g_{i_n}R_{i_n}=R_n$. Geodesic segments in the Cayley graph of $G$ will be denoted
by $[E,F]$.
Hence there exists a sequence $\{R_n,  A_n, B_n, C_n, D_n \}$  such  that (see Figure 1)

\begin{enumerate}
\item $R_n \in \RR$
\item The scaling sequence $d_n = O(dia(R_n)) = O(\rho_n) = O(\delta_n)$, where $\delta_n$ is the hyperbolicity constant of $G(n)$
(The last equality follows from Lemma \ref{delta} and the asymptotic visibility of the sequence $\{ R_n\})$. 
\item  $[A_n, B_n] \subset K_2$; 
\item $[D_n, C_n] \subset R_n$
\item $[A_n, D_n], [B_n, C_n]$ are geodesics in normal form in $G$.
\item $|[A_n, D_n]|, |[B_n, C_n]|$ are $o_\omega(d_n)$
\item $|[A_n, B_n]|, |[D_n, C_n]|$ are $O_\omega(d_n)$
\item $lim_\omega (A_n) = lim_\omega (D_n) =  p$
\item $lim_\omega (B_n) = lim_\omega (C_n) =  q$
\end{enumerate}

\begin{figure}\label{quadl}
\begin{tikzpicture}[thick]
    \path[draw] (-4,0)  coordinate %[label=below: $p_n$] (p_n)
            
            -- ( 4,0)  coordinate [label=right: $K_2$] (K_2) ;
    
    %\draw (-4,0)..controls (-1,5.5) and (1,5.5)..(4,0);
    \node[draw,shape=ellipse, minimum height=1.6cm, minimum width=5cm] (Ell) at (0,1.25) {};
    \draw (-2,0.75) coordinate[label=left: $\scriptstyle D_n$] (D_n)
               -- (-2,0) coordinate[label=below: $\scriptstyle A_n$] (A_n);
    \draw (2,0.75) coordinate[label=right: $\scriptstyle C_n$] (C_n)
               -- (2,0) coordinate[label=below: $\scriptstyle B_n$] (B_n);

    %\draw (0,-0.3) node {$\overline{p_nq_n}$};    
    %\draw (2.7,2.6) node {$\gamma_n$};
    \draw (0,1.25) node {$R_n$};
    %\foreach \point in {p_n,q_n}
           %\fill [black] (\point) circle (2pt);
    \foreach \point in {A_n,B_n,C_n,D_n}
           \fill [black] (\point) circle (1.5pt);       
    
\end{tikzpicture}
 \caption{}
\end{figure}

By cutting off initial and final pieces from  $[A_n, B_n], [D_n, C_n]$ of length $o(d_n)$ if necessary and using thinness of quadrilaterals we can assume
without loss of generality that $|[A_n, D_n]|, |[B_n, C_n]|$ are at most $O(d_{n-1})=O(\delta_{n-1})$. More precisely, choose a subsegment 
$[A_n^\prime, B_n^\prime]$ such that $d(A_n, A_n^\prime) = d(B_n, B_n^\prime) = 2 max \{|[A_n, D_n]|, |[B_n, C_n]|\}$. Then there exist points 
$D_n^\prime, C_n^\prime$ on $[D_n, C_n]$ such that $d(D_n^\prime, A_n^\prime) \leq 2 \delta_{n-1}$ and $d(C_n^\prime, B_n^\prime) \leq 2 \delta_{n-1}$.
Rename the four vertices $A_n^\prime, B_n^\prime, C_n^\prime, D_n^\prime$ as $A_n, B_n, C_n, D_n$.
Then as in the proof of Condition $RSC_3$ in the proof of Theorem \ref{g-gsc} we get a genuine (in the usual small cancellation
sense) piece of length $O(|[A_n, B_n]|)$ 
as a subword of $[A_n, B_n]$. Further, any maximal piece of $R_n$ is $o(\rho_n)$ and hence 
  $\frac{|[A_n, B_n]|}{\rho_n} \rightarrow 0$ as $n \rightarrow \infty$. On the other hand, since 
$lim_\omega (A_n) = lim_\omega (D_n) =  p$
and $lim_\omega (B_n) = lim_\omega (C_n) =  q$, and since $p \neq q$, it follows that $\frac{|[A_n, B_n]|}{\rho_n} $ is bounded away from zero.
This is a contradiction.
\end{proof}
\begin{comment}
Passing to a further subword of $[D_n, C_n]$ if necessary we can assume that $Lab([D_n, C_n]) \in F_2(a,b)$ as the label of $[D_n, C_n]$ must contain at least half
the label in either $F_2(a,b)$ or in $F_2(x,y)$.

The rest of the proof is similar to that of Property $RSC_3$ in Theorem \ref{g-gsc}.
For concreteness  put each of $[A_n, D_n], [B_n, C_n]$  in reduced normal form. 
 that the normal form for the path $[A,B]$ has normal length zero, whereas the normal form for the path $[ADCB]$
has normal length at least one (as the segment $[DC]$ is long and does not lie in $K_{n-1}$ even after cancelling off small bits of length $O(d_{n-1})$
from the beginning and end.

\end{comment}

\section{Controlled Floyd Separation}\label{boundary}
The main purpose of this section is to show:

\begin{theorem} For any Floyd function $f$, $c \geq 1$, and $p, q \in \partial K_2 \subset Pre(\partial)  G$, the points $p, q$ are Floyd separated with respect to 
$c-$quasigeodesics.  \label{nontrivial} \end{theorem}

\begin{proof}
Let $p,q \in \partial K_2$ be distinct points. Let $\{p_n\}_{n=1}^{\infty}$ and 
$\{q_n\}_{n=1}^{\infty}$ be Cauchy sequences in $K_2$, converging to $p$ and $q$ respectively.  We shall show that $d_{f,c}(p, q) > 0$.

Suppose not. Then there exists  a Floyd function $f$ and a constant $c \geq 1$ such that  $d_{f,c}(p, q) = 0$.  Hence there exists 
a sequence of  $(c,c)$-quasi geodesic paths  $\{\gamma_n\}$,
joining 
$q_n$ to $p_n$ in $G$, such that (the $c-$controlled Floyd lengths) $l^{c}_{f}(\gamma_n) \rightarrow 0$.

Let  $D_n$ be the Dehn-diagram obtained from the trivial 
word with label $[p_n,q_n].\gamma_n$ in $G$. We will 
construct a sequence $\{D_n^{\prime}\}$ of subdiagrams  from $\{D_n\}$ and a sequence of numbers $M_n$
  such that 

\begin{enumerate}
\item  $M_n$ is the maximum possible distance of a point $t_n$ on $\gamma_n$, 
from $[p_n,q_n]$.
\item The ultralimit
$\lim^{\omega}(\partial D_n^{\prime}, \{M_n\})$ contains a circle $C$ with non-zero but finite length in the asymptotic cone 
$Con^{\omega}(G,\{M_n\})$.
\item The intersection  $C \cap \TT$ has nonzero length.
\end{enumerate}
 This will contradict Theorem \ref{trivialintn}.

\begin{figure}[!ht]
\label{Dehn diagram2}
\begin{tikzpicture}[thick]
    \path[draw] (-4,0)  coordinate [label=below:$p_n$] (p_n)
            
            -- ( 4,0)  coordinate [label=below:$q_n$] (q_n) ;
    
    \path[draw] (-4,0)..controls (-1,5.5) and (1,5.5)..(4,0);%(quasi)
    
    \draw (-2.5,2.37)  coordinate [label=left:$p_n^{\prime}$] (p_n')
          -- (-2.5,0)  coordinate [label=below:$p_n^{\prime \prime}$] (p_n'');
          
    \draw (2.5,2.37)  coordinate [label=right:$q_n^{\prime}$] (q_n')
          -- (2.5,0)  coordinate [label=below:$q_n^{\prime \prime}$] (q_n'');

    \draw (0,4.12)  coordinate [label=above:$t_n$] (t_n)
         -- (0,0)  coordinate  (s_n);    
  
    %\fill[gray!50] (p_n') -- (quasi) -- (q_n'); 
    
    %\node[draw,shape=ellipse, minimum height=1.6cm, minimum width=5cm] (Ell) at (0,1.25) {};
    %\draw (-2,0.75) coordinate[label=left: $\scriptstyle D_n$] (D_n)
               %-- (-2,0) coordinate[label=below: $\scriptstyle A_n$] (A_n);
    %\draw (2,0.75) coordinate[label=right: $\scriptstyle C_n$] (C_n)
               %-- (2,0) coordinate[label=below: $\scriptstyle B_n$] (B_n);
     \draw (0,-0.3) node {$s_n$};              
    \draw (1.0,-0.3) node {$[{p_nq_n}]$};    
    \draw (3.7,1.0) node {$\gamma_n$};
    %\draw (0,1.25) node {$R_n$};
    \draw (0.3,2) node {$\scriptstyle M_n$};
    \draw (-2.2,1.3) node {$\scriptstyle \leq M_n$};
    \draw (2.2,1.3) node {$\scriptstyle M_n \geq$};
    \draw (-0.3,3) node {$\scriptstyle D_n^{\prime}$};
    \draw (2.2,3.5) node {$D_n$};

    \foreach \point in {p_n',p_n''}
           \fill [black] (\point) circle (1.5pt);
    \foreach \point in {p_n,q_n}
           \fill [black] (\point) circle (2pt);
    \foreach \point in {t_n,s_n}
           \fill [black] (\point) circle (1.5pt);
    
    \foreach \point in {q_n',q_n''}
           \fill [black] (\point) circle (1.5pt);
    
    %\foreach \point in {A_n,B_n,C_n,D_n}
           %\fill [black] (\point) circle (1.5pt);       

\end{tikzpicture}
\caption{$D_n$ and $D_n^{\prime}$}
\end{figure}

Choose  $t_n \in \gamma_n$ such that $d(t_n,[p_n,q_n])=M_n$ is maximal and let $s_n \in [p_n,q_n]$
be such that $d(t_n,s_n)=M_n$. Since $l^{c}_{f}(\gamma_n) \rightarrow 0$ as $n \rightarrow \infty$, it follows that
$M_n \rightarrow 0$ as $n \rightarrow \infty$. The actual dependence of $M_n$ on $f, c$ is not important as we shall only use the sequence $\{ M_n \}_n$
as a sequence of scale factors to extract a limiting asymptotic cone.

The proof now splits into the following cases:

\begin{enumerate}
\item[Case 1:] The  subarcs of $\gamma_n$ joining 
 $p_n,t_n$ and  $t_n,q_n$ each have  length at least $5cM_n$.
\item[Case 2:] The  subarcs of $\gamma_n$ joining 
 $p_n,t_n$ and  $t_n,q_n$ each have  length less than $5cM_n$.
\item[Case 1:] Exactly one of the  subarcs of $\gamma_n$ joining 
 $p_n,t_n$ and  $t_n,q_n$  has  length at least $5cM_n$.
\end{enumerate}

We shall prove the Theorem for Cases 1 and 2. We shall only give a sketch of Case 3, which 
 is a hybrid and the  proofs of Cases 1 and 2 combine in an obvious way.

\smallskip

\noindent {\bf Case 1:}\\
Choose  $p_n^{\prime}, q_n^{\prime}$ to the left and right
respectively of $t_n$ such that the subarcs of $\gamma_n$ joining 
 $p_n^{\prime},t_n$ and  $t_n,q_n^{\prime}$ each have  length $5cM_n$. The points $p_n^{\prime}, q_n^{\prime}$ exist by rectifiability and connectedness
of the arc $\gamma_n$. It follows that the 
 $c$-quasi geodesic arc ($c>1$) joining $p_n^{\prime}$ and $q_n^{\prime}$ has length $10cM_n$ and hence
 $d(p_n^{\prime},q_n^{\prime}) \geq 10M_n-c$.

Let $p_n^{\prime\prime}$ (resp.
 $q_n^{\prime\prime}$) be points on $[p_n,q_n]$  closet to $p_n^{\prime}$ (resp. $q_n^{\prime}$).
 By the choice of $t_n$,  $d(p_n^{\prime},p_n^{\prime\prime}) \leq M_n$ and 
 $d(q_n^{\prime},q_n^{\prime\prime})\leq M_n$. Hence 
 $|[p_n^{\prime\prime},q_n^{\prime\prime}]| \geq 8M_n-c$.
 
Let $D_n^{\prime}$ be the subdiagram of $D_n$ with sides 
 $[{p_n^{\prime},p_n^{\prime\prime}}]$, $[{p_n^{\prime\prime},q_n^{\prime\prime}}]$, 
 $[{q_n^{\prime\prime},q_n^{\prime}}]$ and the subarc of $\gamma_n$ joining $q_n^{\prime}$ and 
 $p_n^{\prime}$. 
 
 Let $D^\prime = \lim^{\omega} \partial D_n^{\prime} \subset Con^{\omega}(G,\{M_n\}); p^\prime = \lim^{\omega}p_n^{\prime};  
q^\prime = \lim^{\omega}q_n^{\prime}; p^{\prime\prime} = \lim^{\omega}p_n^{\prime\prime};
q^{\prime\prime} = \lim^{\omega}q_n^{\prime\prime}; t = \lim^{\omega}t_n$. Then 
 $|[{p^{\prime\prime},q^{\prime\prime}}]| \geq 8$, $|[p^{\prime}, {p^{\prime\prime}}]| \leq 1$,
$|[{q^{\prime},q^{\prime\prime}}]| \leq 1$.  Also the top arc $\eta$ of $D^{\prime}$, joining $p^\prime, t, q^\prime$
 is an {\bf embedded} ($c-$biLipschitz) arc of length $10c$ in $Con^{\omega}(G,\{M_n\})$. Hence
$D^{\prime}$ contains a loop $\theta$ such that $10c-2 \leq |\theta \cap \eta| \leq 10c$ and $6 \leq |[p^{\prime},q^{\prime}]|$. 
 Since $Con^{\omega}(G,\{M_n\})$ is a circle tree, and since $\eta, [p^{\prime\prime},q^{\prime\prime}]$ are both embedded arcs,
it follows that $\theta$ contains a simple loop $C$ (a circle) whose intersection with  $ [p^{\prime\prime},q^{\prime\prime}]$ has non-zero length.
 This  contradicts Theorem \ref{trivialintn}.\\

\noindent {\bf Case 2:}\\ In this case, since $M_n \leq |\gamma_n| + d (p,q)$ (where $|\gamma_n|$ denotes the usual length of $\gamma_n$), it follows
that both $|\gamma_n|$ and $d (p,q)$ are $O(M_n)$. Next, since $\gamma_n$ is a $(c,c)-$quasigeodesic, $lim^\omega (\gamma_n)$ is a $c-$biLipschitz path.
In particular, $lim^\omega (\gamma_n)$ has no self-intersections. Hence, $D = \lim^{\omega} \partial D_n \subset Con^{\omega}(G,\{M_n\})$ is a biLipschitz bigon
with $p := \lim^{\omega}p_n$ and  
$q := \lim^{\omega}q_n$ as its two vertices. Since $t := \lim^{\omega}t_n$ lies at distance $1$ from $[p,q]$, it follows that there is a nontrivial bigon $B$
with bounding arcs $\alpha, \beta$ such that 
\begin{enumerate}
\item the interiors of $\alpha, \beta$ do not intersect,
\item $\alpha$ contains $t$ and is a nontrivial subarc of $\lim^{\omega}{\gamma}_n$,
\item $\beta$  is a nontrivial subarc of $[p,q]$.
\end{enumerate}
Again, this  contradicts Theorem \ref{trivialintn}.\\

\noindent {\bf Case 3:}\\ Without loss of generality, suppose that  the  subarc of $\gamma_n$ joining 
 $p_n,t_n$  has  length at least $5cM_n$ and the subarc of $\gamma_n$ joining 
 $q_n,t_n$  has  length less than $5cM_n$. Construct $p_n^\prime, p_n^{\prime \prime}$ as in Case 1 and let $D_n^\prime$ be the subdiagram of $D_n$
bounded by $[p_n^\prime, p_n^{\prime \prime}]$, $[ p_n^{\prime \prime}, q_n]$ and the subarc of $\gamma_n$ joining $q_n$ to $p_n^{\prime \prime}$.
Scaling by $M_n$ and taking limits we again obtain a nontrivial bigon  contradicting Theorem \ref{trivialintn}.
\end{proof}
 
We isolate now  a notion of `quasiconvexity with respect to quasigeodesics'  (slightly stronger than usual quasiconvexity) that
the {\it proof } of the above Theorem
furnishes for $K_2$:

\begin{defn} \label{qcqg} Let $G$ be a finitely generated group and $\Gamma$ a Cayley graph with respect to a finite generating set. 
For any $c \geq 1$, a  finitely generated subgroup $H \subset G$
is said to be {\bf quasiconvex with respect to $c-$quasigeodesics} if  there exists $D >0$ such that for any $h_1, h_2 \in H$, any $(c,c)-$quasigeodesic
in $\Gamma$ lies in a $D-$neighborhood of $H$. 

$H \subset G$
is said to be {\bf quasiconvex with respect to quasigeodesics}  if it is 
 quasiconvex with respect to $c-$quasigeodesics for all $c$. \end{defn}

Thus  a subgroup is quasiconvex with respect to quasigeodesics if all geodesics on it are {\it `uniformly Morse'}.

\begin{cor} $K_2 \subset G$ is quasiconvex with respect to quasigeodesics. \label{morse} \end{cor}

\begin{proof} We continue with the notation used in Theorem \ref{nontrivial}.
  Note first that in the proof of Theorem \ref{nontrivial}, the final contradiction only requires that the distance $d(t_n, (p,q)) \rightarrow \infty$
as $n\rightarrow \infty$ for some $c \geq 1$.

Suppose that there exists $c \geq 1$ such that
 $K_2 \subset G$ is not quasiconvex with respect to $c-$quasigeodesics. Then there exist $(c,c)-$quasigeodesics $\gamma_n$ joining $p_n, q_n \in K_2$
such that
$d(\gamma_n, (p,q)) \rightarrow \infty$
as $n\rightarrow \infty$ for some $c > 0$. Hence there exists $h_n \in \gamma_n$ such that 
$d(h_n, K_2) \rightarrow \infty$
as $n\rightarrow \infty$. Translating by an element of $K_2$, we can assume that $d(h_n, K_2)  = d(h_n, 1) $. The argument of 
Theorem \ref{nontrivial} now goes through as before to furnish a circle $C$ with non-zero but finite length in the asymptotic cone 
$Con^{\omega}(G,\{h_n\})$ such that
the intersection  $C \cap \TT$ has nonzero length.  This  contradicts Theorem \ref{trivialintn}.
\end{proof}

\section{Triviality of Floyd Boundary} For completeness, we show in this section that the usual Floyd boundary $\partial_{f} G$ consists of a single point
and is hence trivial:

\begin{theorem} $\partial_{f} G$ consists of a single point. \label{floydtrivial} \end{theorem}

\begin{proof}
Let $\partial_f F_2(a,b) \subset \partial_{f} G$ denote the set of limit points of $F_2(a,b)$ in $ \partial_{f} G$.
  Let $p, q \in \partial_f F_2(a,b)$ and let $(p,q)$ denote the label of the bi-infinite geodesic path between them. 
Assume without loss of generality that $1 \in (p,q)$.
By the construction of the $h_n$'s in Section \ref{eg}, there exist $p_n \rightarrow p, q_n \rightarrow q$ such that 
\begin{enumerate}
\item $[p_n, q_n]$ is centered at $1$
\item the geodesic path $[p_n, q_n]$ (i.e. the word $p_n^{-1} q_n$)
in $a, b$ is a subpath of $[1,h_m]$ (cf. the construction in Section \ref{eg}) 
for some $m=m(n)$, surviving in the construction of $R_m$.
\end{enumerate}

We make  condition (2) more precise.
Choose $h_m$ and $c_n, d_n \in [1,h_m]$ such that 
\begin{enumerate}
\item $c_n^{-1}d_n = p_n^{-1} q_n$ 
\item Any geodesic $[1, k]$ with $k \in K$, containing $[1, c_n]$ as a subsegment  must have length greater than $2|c_n|$, i.e. $|k| > 2 |c_n|$.
Hence, the reduced normal form of $h_m^{-1}(x,y)h_m(a,b)$ contains $c_n^{-1}(a,b)d_n(a,b)$ (and hence also $c_n^{-1}(x,y)d_n(x,y)$) as a subword.
\end{enumerate}
 By ensuring the above, we are guaranteed the existence  of a relator
$R_m (=R_{m(n)}) \in \RR$ such that   $p_n^{-1}(a,b)q_n(a,b)$ (and hence also $p_n^{-1}(x,y)q_n(x,y)$) is a subword of $R_m$. 

By Lemma \ref{delta}, for $n$ sufficiently large, $R_{m(n)} $ can be described as the union of two $(2,2)$ quasigeodesics $[1,t_n]_1$ and 
$[1,t_n]_2$ both starting at $1$ and ending
at an `antipodal' point $t_n \in R_{m(n)} $. Further, $p_n \in [1,t_n]_1$ and $q_n \in [1,t_n]_2$ (See Figure 3).

\begin{figure}\label{contiguity diagram}
\begin{tikzpicture}
\draw[thick] (4,0) arc (-45:225:4);
\draw[thick] (-1.65,0) coordinate[label=left: $p_n$] (p_n) -- 
      (4,0) coordinate[label=right: $q_n$] (q_n);

\draw (1.175,0) coordinate[label=below: $\scriptstyle 1$] (id);

\begin{comment} circle (2.5);

\draw[->] (1.175,0) -- (3.25,1.4);           
\draw (2.5,.6) node {$\scriptstyle \nu_n$};
\end{comment}
\draw (1.175,6.825) coordinate[label=above: $\scriptstyle t_n$] (t_n);

\foreach \point in {p_n,q_n}
           \fill [black] (\point) circle (1.5pt);
\foreach \point in {t_n,id}
           \fill [black] (\point) circle (1.2pt);
\draw (5.1,5.1) node {$R_{m(n)}$};           
\end{tikzpicture}

 \caption{$d_{f}(p_n,q_n)$ is small.}
\end{figure}
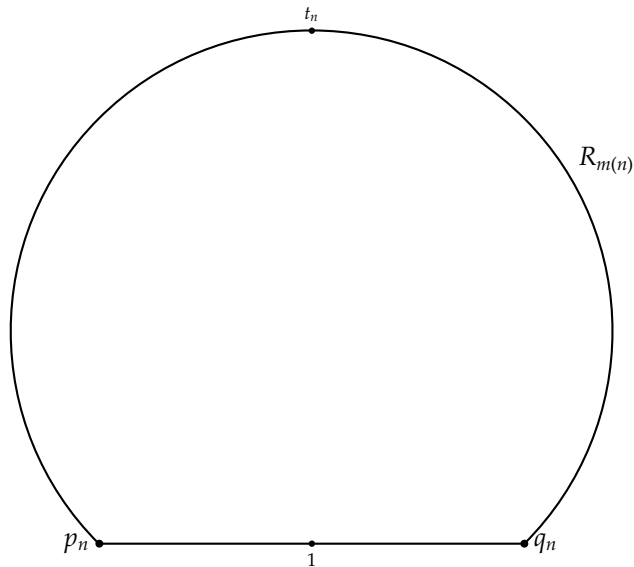

Assume that the Floyd length of the geodesic rays $[1,p)$ (and $[1,q)$) are normalized to one.
Hence for all $\epsilon > 0$, there exists $n$ sufficiently large such that the Floyd lengths  of the subpaths $[p_n,t_n]_1$ and 
$[q_n,t_n]_2$ are each less than $\frac{\epsilon}{2}$.  It follows that $d_f(p_n, q_n) < \epsilon$.

Since $\epsilon > 0$ is arbitrary, $p, q$ correspond to the same point on the Floyd boundary. Hence $\partial_f F_2(a,b)$ 
consists of a single point. Similarly, $\partial_f F_2(x,y)$ consists of a single point $z$. 

Next, from the  relators $R_n$ of the form $w_n(x,y)w_n^{-1}(a,b)$ centered at $1$ one can extract subsequential limits $z_1 \in \partial_f F_2(x,y)$ and $z_2
\in \partial_f F_2(a,b)$ of 
$w_n(x,y)$ and $w_n(a,b)$ to conclude that  $z_1 = z_2$ on $\partial_{f} G$. Hence $\partial_f F_2(x,y) = \partial_f F_2(x,y)=z$. Since the action of $G$ on
$\partial_{f} G$ is minimal \cite[Theorem 2 and Section VI]{karlsson-free}, and since $z$ is invariant under both $F_2(a,b)$ as well as $F_2(x,y)$ and hence under $G$, it follows that 
$\partial_{f} G = \{ z \}$.
\end{proof}

\noindent {\bf Concluding Remarks:}
As pointed out in Remark \ref{dah}, Floyd separation of points of $\partial K_2 \subset Pre (\partial) G$ with respect to quasigeodesics
would follow if $K_2$ was hyperbolically
embedded in $G$. However, this seems quite unlikely for the following reason. The proof of Theorem \ref{nrh} shows that if $G$ admits an action by isometries on a hyperbolic
space $X$ and both $F_2(a,b)$ and $F_2(x,y)$ are qi embedded then the limit sets of $F_2(a,b)$ and $F_2(x,y)$ must coincide. This would mean that we can take $X$
to be quasi-isometric to the Cayley graph of $F_2(a,b)$ (or $F_2(x,y)$). But then all elements of the form $w(x,y) w(a,b)^{-1}$ would lie in
the kernel of the action preventing such an action from being acylindrical in any sense. Perhaps this argument can be strengthened to give a negative answer to the
following:

\begin{qn} \label{hypembed} Is $K_2$ hyperbolically embedded in $G$? \end{qn}

We should point out however that $G$ {\it does} admit an acylindrical action on the Bass-Serre tree $T$ corresponding to the splitting along $K$.
Of course $K$ keeps an edge of $T$ fixed.

\medskip

\noindent {\bf Acknowledgments:} The  authors would like to thank Dani Wise for a number of useful discussions and email exchanges
and particularly for pointing out the reference \cite{wise-qpor} and telling us the proof idea of Proposition \ref{cycgeodprop}.
We are grateful to Victor Gerasimov for pointing out an error in a previous version (See Remark \ref{gerrmk}). We would also like to thank Leonid Potyagailo
and Francois
Dahmani for helpful comments and the referee for a careful reading and several helpful comments on the paper, particularly on Theorem \ref{nontrivial}.

\bibliography{floyd}

\begin{thebibliography}{DGO11}

\bibitem[ACT14]{act}
G.~N. Arzhantseva, C.~H. Cashen, and J.~Tao.
\newblock Growth tight actions.
\newblock {\em preprint, arXiv:1401.0499}, 2014.

\bibitem[Bes04]{bestvinahp}
M.~Bestvina.
\newblock Geometric group theory problem list.
\newblock {\em M. Bestvina's home page: http:math.utah.edu\~bestvina}, 2004.

\bibitem[BF92]{BF}
M.~Bestvina and M.~Feighn.
\newblock A {C}ombination theorem for {N}egatively {C}urved {G}roups.
\newblock {\em J. Diff. Geom., vol 35}, pages 85--101, 1992.

\bibitem[Bow12]{bowditch-relhyp}
B.~H. Bowditch.
\newblock Relatively hyperbolic groups.
\newblock {\em Internat. J. Algebra and Computation. 22, 1250016, 66pp}, 2012.

\bibitem[Dah03]{dahmani-th}
F.~Dahmani.
\newblock Les groupes relativement hyperboliques et leurs bords.
\newblock {\em PhD thesis, Univ. Louis Pasteur, Strasbourg, France. Pr epubl.
  Inst. Rech. Math. Av., Univ. Louis Pasteur, 2003/13.}, 2003.

\bibitem[DGO11]{dgo}
F.~Dahmani, V.~Guirardel, and D.~Osin.
\newblock Hyperbolically embedded subgroups and rotating families in groups
  acting on hyperbolic spaces.
\newblock {\em preprint, arXiv:1111.7048}, 2011.

\bibitem[DS05]{ds}
C.~Drutu and M.~Sapir.
\newblock Tree-graded spaces and asymptotic cones of groups.
\newblock {\em Topology 44}, pages 959--1058, 2005.
\newblock With an appendix by Denis Osin and Sapir.

\bibitem[Far98]{farb-relhyp}
B.~Farb.
\newblock Relatively hyperbolic groups.
\newblock {\em Geom. Funct. Anal. 8}, pages 810--840, 1998.

\bibitem[Flo80]{floyd}
W.~J. Floyd.
\newblock Group {C}ompletions and {L}imit {S}ets of {K}leinian {G}roups.
\newblock {\em Invent. Math. vol.57}, pages 205--218, 1980.

\bibitem[Ger96]{gersubgp}
S.~Gersten.
\newblock Subgroups of hyperbolic groups in dimension 2.
\newblock {\em J. London Math. Soc.}, pages 261--283, 1996.

\bibitem[Ger09]{gerasimov-expans-gafa}
V.~Gerasimov.
\newblock Expansive convergence groups are relatively hyperbolic.
\newblock {\em Geom. Funct. Anal.}, 19:137--169, 2009.

\bibitem[Ger12]{ger-floydgafa}
V.~Gerasimov.
\newblock {Floyd maps for relatively hyperbolic groups}.
\newblock {\em Geom. Funct. Anal. 22 no. 5}, pages 1361--1399, 2012.

\bibitem[GM08]{groves-manning}
D.~Groves and J.~Manning.
\newblock Dehn filling in relatively hyperbolic groups.
\newblock {\em Israel Journal of Mathematics 168}, pages 317--429, 2008.

\bibitem[GP11]{ger-pot-qc}
V.~Gerasimov and L.~Potyagailo.
\newblock {Quasiconvexity in the Relatively Hyperbolic Groups}.
\newblock {\em preprint, arXiv:1103.1211}, 2011.

\bibitem[GP13]{ger-pot-jems}
V.~Gerasimov and L.~Potyagailo.
\newblock { Quasi-isometric maps and Floyd boundaries of relatively hyperbolic
  groups}.
\newblock {\em J. Eur. Math. Soc. 15, no. 6}, pages 2115--2137, 2013.

\bibitem[Gro85]{gromov-hypgps}
M.~Gromov.
\newblock Hyperbolic {G}roups.
\newblock {\em in Essays in Group Theory, ed. Gersten, MSRI Publ.,vol.8,
  Springer Verlag}, pages 75--263, 1985.

\bibitem[Gro93]{gromov-ai}
M.~Gromov.
\newblock Asymptotic {I}nvariants of {I}nfinite {G}roups.
\newblock {\em in Geometric Group Theory,vol.2; Lond. Math. Soc. Lecture Notes
  182, Cambridge University Press}, 1993.

\bibitem[Hru10]{hruska-agt}
G.~Christopher Hruska.
\newblock {Relative hyperbolicity and relative quasiconvexity for countable
  groups}.
\newblock {\em {Algebr. Geom. Topol. 10 no. 3}}, pages 1807--1856, 2010.

\bibitem[Kar03]{karlsson-free}
A.~Karlsson.
\newblock Free subgroups of groups with nontrivial {F}loyd boundary.
\newblock {\em Communications in Algebra}, 31(11):5361--5376, 2003.

\bibitem[LS]{ls}
R.~C. Lyndon and P.~E. Schupp.
\newblock Combinatorial {G}roup {T}heory.
\newblock {\em Springer}.

\bibitem[Mit04]{mitra-ht}
M.~Mitra.
\newblock Height in {S}plittings of {H}yperbolic {G}roups.
\newblock {\em Proc. Indian Acad. of Sciences, v. 114, no.1}, pages 39--54,
  Feb. 2004.

\bibitem[Mj08]{mahan-relrig}
M.~Mj.
\newblock {Relative Rigidity, Quasiconvexity and C-Complexes}.
\newblock {\em Algebraic and Geometric Topology 8}, pages 1691--1716, 2008.

\bibitem[Ols91]{olshanski-book}
A.~Yu. Olshanskii.
\newblock {\em Geometry of defining relations in groups}.
\newblock Translated from the 1989 Russian original by Yu. A. Bakhturin. 70.
  Kluwer Academic Publishers Group, Dordrecht, 1991.
\newblock Mathematics and its Applications (Soviet Series).

\bibitem[OOS09]{oos}
A.~Yu. Olshanskii, D.~V. Osin, and M.~V. Sapir.
\newblock Lacunary hyperbolic groups.
\newblock {\em Geom. Topol.}, 13(4):2051--2140, 2009.
\newblock With an appendix by Michael Kapovich and Bruce Kleiner.

\bibitem[Osi06]{osin-mams}
D.~V. Osin.
\newblock Relatively hyperbolic groups: intrinsic geometry, algebraic
  properties, and algorithmic problems.
\newblock {\em Mem. Amer. Math. Soc.}, 179(843):vi+100, 2006.

\bibitem[Wis01]{wise-por}
D.~T. Wise.
\newblock The residual finiteness of positive one-relator groups.
\newblock {\em Comment. Math. Helv.}, 76(2):314--338, 2001.

\bibitem[Wis02]{wise-qpor}
D.~T. Wise.
\newblock Residual finiteness of quasi-positive one-relator groups.
\newblock {\em J. London Math. Soc. (2) 66}, pages 334--350, 2002.

\bibitem[Wis12]{wise-cbms}
D.~T. Wise.
\newblock From riches to raags: 3-manifolds, right-angled artin groups, and
  cubical geometry.
\newblock {\em American Mathematical Soc.}, 2012.

\bibitem[Yan12]{yang-relhyp}
W.~Yang.
\newblock Limit sets of relatively hyperbolic groups.
\newblock {\em Geom. Dedicata 156}, pages 1--12, 2012.

\end{thebibliography}
\bibliographystyle{alpha}

\end{document}